\documentclass[11pt]{article}
\usepackage[utf8]{inputenc}  \usepackage[T1]{fontenc} 
\usepackage{lmodern}
\usepackage{array} 
\usepackage{paralist} 
\usepackage{amsmath} \usepackage{amsthm} 
\usepackage[british]{babel}  \usepackage{amssymb}  \usepackage{graphicx}
\usepackage[left=2.5cm,top=2.5cm,right=2.5cm,bottom=2cm]{geometry}  \usepackage{color} 
\usepackage{comment}

\usepackage[toc]{appendix}

\newtheorem{theorem}{Theorem}[section]
\newtheorem{lemma}[theorem]{Lemma}

\newtheorem{definition}[theorem]{Definition}
\theoremstyle{definition}
\newtheorem{remark}[theorem]{Remark}

\newtheorem{notation}[theorem]{Notation} 
\numberwithin{equation}{section}

\newcommand{\R}{\mathbb{R}}

\newcommand{\Z}{\mathbb{Z}}
\newcommand{\C}{\mathbb{C}} 
\newcommand{\T}{\mathbb{T}} 
\newcommand{\N}{\mathbb{N}} 
\newcommand{\MAH}{\mathbb{H}}

\newcommand{\MS}{\mathcal{S}}
\newcommand{\MT}{\mathcal{T}}

\newcommand{\MN}{\mathcal{N}}
\newcommand{\MK}{\mathcal{K}}
\newcommand{\MR}{\mathcal{R}}
\newcommand{\MG}{\mathcal{G}}

\newcommand{\MY}{\mathcal{Y}}
\newcommand{\MZ}{\mathcal{Z}}
\newcommand{\MW}{\mathcal{W}}

\newcommand{\wh}{\widehat}
\newcommand{\ol}{\overline}
\newcommand{\wt}{\widetilde}

\newcommand{\Lip}{\mathrm{Lip\,}}

\title{Invariant manifolds of maps and vector fields with nilpotent parabolic tori}

\begin{document}
	 \date{}
	\maketitle
	
	\vspace{-1.5cm}
	 
	\centerline{Clara Cufí-Cabré}
	\medskip
	{\footnotesize
		\centerline{Departament de Matemàtiques} 
		\centerline{Universitat Autònoma de Barcelona (UAB)} 
		\centerline{08193 Bellaterra, Barcelona, Spain}
		\centerline{\texttt{clara.cufi@uab.cat}}
	} 
	
	\medskip
	\vspace{8pt}
	
	\centerline{Ernest Fontich}
	\medskip
	{\footnotesize
		\centerline{Departament de Matemàtiques i Informàtica} 
		\centerline{Universitat de Barcelona (UB)}
		\centerline{Centre de Recerca Matemàtica (CRM)}
		\centerline{Gran Via 585, 08007 Barcelona, Spain}
		\centerline{\texttt{fontich@ub.edu}}
	}
	
	\bigskip

	\begin{abstract}
We consider analytic maps and vector fields  defined in $\R^2 \times \T^d$, having a $d$-dimensional invariant torus $\MT$. The map (resp. vector field) restricted to $\MT$ defines a rotation of frequency $\omega$, and its derivative restricted to transversal directions to $\MT$ does not diagonalize. In this context, we give conditions on the coefficients of the nonlinear terms of the map (resp. vector field)  under which $\MT$ possesses stable and unstable invariant manifolds, and we show that such invariant manifolds are analyitic away from the invariant torus. We also provide effective algorithms to compute approximations of  parameterizations of the invariant manifolds, and present some applications of the results.
	\end{abstract}
	\vspace{0.1cm}
	
	\begin{footnotesize}
	\end{footnotesize}
	
	
	\section{Introduction}

The study of parabolic invariant manifolds is relevant, apart from the interest that presents itself as a mathematical problem, because this kind of manifolds appears naturally in many problems motivated by physics, chemistry and other sciences.

Parabolic manifolds have been used to prove the existence of oscillatory motions in some well-known problems of Celestial Mechanics as  the Sitnikov problem \cite{sit60, mos73} and the circular planar restricted three-body problem
\cite{gms16,gsms17,llibresimo80} using the transversal intersection of invariant manifolds of parabolic points and symbolic dynamics. 

The existence of oscillatory motions in all these instances is strongly related to  some invariant objects at infinity that are either fixed points, periodic orbits or invariant tori and to their stable and unstable manifolds. Although these invariant objects are parabolic in the sense that the linearization of the vector field on them has all the eigenvalues equal to zero, they do have stable and unstable invariant manifolds in a similar way as for hyperbolic invariant objects, that is, 
the sets of initial conditions of solutions that tend to the invariant object when $t\to \pm \infty$, for the stable/unstable manifolds, respectively.

Let $\T^d = (\R / \Z)^d$ be the real torus of dimension $d$ and $U \subset \R^2$. In this paper, we consider analytic maps $F: U \times \T^d \times \Lambda \to \R^2 \times \T^d$  of the form 
\begin{equation} \label{forma-general-maps}
	F(x, y, \theta, \lambda) = 
	\begin{pmatrix}
		x + c(\theta, \lambda) y + f_1(x, y, \theta, \lambda) \\
		y +  f_2(x, y, \theta, \lambda)  \\
		\theta + \omega + f_3(x, y, \theta, \lambda)
	\end{pmatrix},
\end{equation}
where $(x, y) \in U \subset \R^2$, $\theta \in \T^d$, $\omega \in \R^d$, $\lambda \in \Lambda \subset \R^m$, with $f_1(x, y, \theta, \lambda), f_2(x, y, \theta, \lambda) = O(\|(x, y)\|^2)$, and $f_3(x, y, \theta, \lambda) = O(\|(x, y)\|)$, 
and analytic vector fields 
$X: U \times \T^d \times \R \times \Lambda \to \R^2 \times \T^d$  of the form 
\begin{equation} \label{forma-general-edo}
	X(x, y, \theta, t, \lambda) =
	\begin{pmatrix}
		c(\theta, t, \lambda) y + g_1(x,y,\theta, t, \lambda) \\
		g_2(x,y,\theta, t, \lambda)  \\
		\omega + g_3(x,y,\theta, t, \lambda)
	\end{pmatrix},
\end{equation}
with $g_1(x, y, \theta, \lambda), g_2(x, y, \theta, \lambda) = O(\|(x, y)\|^2)$, and $g_3(x, y, \theta, \lambda) = O(\|(x, y)\|)$, depending quasiperiodically on the time variable $t$.

The set 
$$
\mathcal{T} = \{ (0,0,\theta) \in U \times \T^d  \}
$$ 
is an invariant torus of $F$ and  $X$, that is, for all $\lambda \in \Lambda$,  $F(\MT, \lambda) \subset \MT$, or in the case of vector fields, for any point $x \in \MT$, $X(x, \lambda)$ is tangent to $\MT$ at $x$. We say that $\MT$ is a \emph{parabolic torus with nilpotent part} because the top-left $2 \times 2$ box of the derivative of $F$ (resp. $X$) at  $(0, 0, \theta)$ does not diagonalize.

We study the existence and regularity of $(d+1)$-dimensional invariant manifolds of such maps and vector fields. Concretely, we give sufficient conditions on the coefficients of $F$ and $X$ that guarantee the existence of a stable and an unstable invariant manifold, and we prove that such invariant manifolds are analytic away form $\MT$. Moreover, we provide an algorithm to compute an approximation of a parameterization of the invariant manifolds up to any order. The results also provide  analytic dependence on parameters.

For the case of maps  we use a method similar to the one in \cite{BFM-20-w}, where the authors study  the existence of invariant manifolds of analytic maps and vector fields defined on $\R^n \times \T^d$, and where the first $n \times n$ box of the linear part of the maps is equal to the identity. There, applications to the study of the planar $(n+1)$-body problem are provided. Contrary to that paper, here the results for vector fields are presented as a direct study of  a functional equation in a suitable Banach space, while in \cite{BFM-20-w} the corresponding results are obtained from the results for maps.

In  \cite{zhang16} the authors deal with invariant curves of $C^\infty$ planar maps of the form corresponding to the first two components of \eqref{forma-general-maps}  and they obtain the existence of a local stable manifold  as the graph of some function $\varphi$ by solving a fixed point equation equivalent to the invariance of the graph of $\varphi$. This equation is solved by applying the Schauder fixed point theorem, and they obtain invariant manifolds of class $C^{[(k+1)/2]}$, where in our notation such $k$ appears in the reduced form of $F$ given in Section \ref{sec-reduced}. In this paper we generalize the results of \cite{zhang16} in several ways. First, we consider both maps and vector fields having a parabolic invariant torus, where the particular case of the maps restricted to the directions transversal to the torus have the form considered in \cite{zhang16}. Also, the maps and vector fields that we consider are analytic, and we provide the existence of analytic invariant manifolds for them, away from the invariant tori. 
We also provide a result of uniqueness for the stable manifold for maps.
In our approach we use the parameterization method, which allows us to state the existence theorems as \emph{a posteriori} results, and we provide explicit algorithms to compute an approximation of a parameterization of the invariant manifolds. Moreover, the use of this technique and the Banach fixed point theorem allows to prove the existence results in a more compact way than in \cite{zhang16}.

The results of \cite{CC-EF-20} also generalize in some ways the ones of \cite{zhang16}. There, planar maps of class $C^r$ having a nilpotent parabolic point are considered, and the existence of $C^r$ invariant curves asymptotic to the fixed point is provided. The present paper is indeed a natural generalization of  \cite{CC-EF-20} in the analytic case.

The paper is organized as follows. In Section \ref{sec-reduced} we introduce the \emph{reduced form} of the maps and vector fields we deal with and we present the parameterization method. Next we state the main results concerning the existence of invariant manifolds (Theorems \ref{teorema_analitic_tors} - \ref{teorema_posteriori_tors_edo}) and we present applications  to the study of a quasiperiodically forced oscillator and  the scattering of He atoms off Cu surfaces in Section \ref{sec-appl-tors}.  
In Section \ref{sec-algoritme-tors}, we provide an algorithm to compute an approximation of a parameterization of the invariant manifolds, both for maps and vector fields. The main results of that section are Theorems \ref{prop_simple_tors} and \ref{prop_tors_edo_sim}. The rest of the paper is devoted to prove the main results of existence. First, in Section \ref{sec-funcional} we introduce a functional equation equivalent to the invariance of a parameterized manifold which will be the object of our study. We also introduce suitable function spaces, some operators, and state their properties. Finally, the proofs of the main results are provided in Section \ref{sec-dems-tors}.
In the Appendix we give two postponed proofs and we deal with unstable manifolds. 

	\section{Main results and applications}
	
In this section we state the main results of the paper. 
To simplify the statements we write them for maps and vector fields in their reduced form, introduced below. This requires some preliminary notation. After the statements of the results we provide some applications.

Let $V \subset \R^n$ be an open set. 
We say that a function $h: V \times \R \to \R^\ell$, $h=h(x, t)$, depends \emph{quasiperiodically} on $t$  if there exists a vector $\nu = (\nu_1 , \, \dots , \, \nu_{d'}) \in \R^{d'}$ and a function $\check h : U \times \T^{d'} \to \R^\ell$, called the hull function of $h$,  such that 
$$
h(x, t) = \check h(x, \nu t).
$$
We call $\nu$ the vector of \emph{time frequencies} of $h$. If $d'=1$ then $h$ is a periodic function of $t$.

Given a map  $h : V\times \T^d  \to \R^\ell$, we define the \emph{average} of $h$ with respect to $\theta \in \T^d$ as
$$
\bar{h}(x) = \frac{1}{\text{vol}(\T^d)} \int_{\T^d} h(x, \theta) \, d\theta,
$$
and the \emph{oscillatory} part of $h$ as
$
\tilde{h}(x,\theta) = h( x,\theta) - \bar{h}(x)$. Sometimes, if a function $h$ does not depend on $\theta$, we will still write $\bar h$ to emphasize this fact. If $h$ also depends quasiperiodically on time, we write
$$
\bar{h}(x) = \frac{1}{\text{vol}(\T^d\times \T^{d'})} \int_{\T^d\times \T^{d'}} \check  h(x, \theta,\theta') \, d\theta d\theta'.
$$
If $t$ denotes the time variable, then given two functions $g(x,t)$ and $h(x,t)$ the composition $f = h \circ g$ will mean
$$
f(x, t) = h(g(x, t), t).
$$

We will deal with functions depending on a paramater $\lambda \in \Lambda\subset \R^m$. The previous definitions and notation extend naturally to such functions.

\subsection{Reduced form of the maps and vector fields} \label{sec-reduced}

Along this paper we consider analytic maps of the form \eqref{forma-general-maps} and analytic vector fields of the form \eqref{forma-general-edo}. Performing the analytic change of variables given by $\tilde x = x$,   $\tilde{y} = y + \frac{1}{c(\theta, \lambda)} \, f_1(x,  y, \theta, \lambda)$, $\tilde \theta = \theta $, the nonlinear terms of the first components of \eqref{forma-general-maps} and \eqref{forma-general-edo} are removed, and there remain only the linear terms. Performing also the change $y \mapsto -y$ if necessary, \eqref{forma-general-maps} and \eqref{forma-general-edo} can be written respectively as
\begin{equation} \label{forma_normal_tor} 
	F(x, y, \theta, \lambda) = 
	\begin{pmatrix}
		x + c(\theta, \lambda) y \\
		y + a_k(\theta, \lambda)x^k + A(x, y, \theta, \lambda)  \\
		\theta + \omega +d_p(\theta, \lambda)x^p + B(x, y, \theta, \lambda)
	\end{pmatrix}
\end{equation}	
and
\begin{equation} \label{fnormal_tors_edo1}
	X(x, y, \theta, t, \lambda) =
	\begin{pmatrix}
		c(\theta, t, \lambda) y \\
		a_k(\theta, t, \lambda)x^k + A(x, y, \theta, t, \lambda)  \\
		\omega + d_p(\theta, t, \lambda)x^p + B(x, y, \theta, t, \lambda)
	\end{pmatrix},
\end{equation}
with $(x,y, \theta, \lambda) \in U \times \T^d  \times \Lambda$ and $\bar c >0$, for some $k \geq 2$, $p \geq 1$. We assume that   
\begin{align} \begin{split} \label{condicions_fnormal}
		A(x, y, \theta,t, \lambda) = y \,  O(\|(x, y)\|^{k-1}) +  O(\|(x, y)\|^{k+1}),  \\
		B(x, y, \theta, t, \lambda) = y \,  O(\|(x, y)\|^{p-1}) + O(\|(x, y)\|^{p+1}),
\end{split} \end{align}
without the time dependence in the case of maps. We also assume that the vector field $X$ given in \eqref{fnormal_tors_edo1} depends quasiperiodically on $t$, being $\nu \in \R^{d'}$ the vector of time frequencies.

From now on, we will refer to \eqref{forma_normal_tor} and \eqref{fnormal_tors_edo1} as the \emph{reduced form} of \eqref{forma-general-maps} and \eqref{forma-general-edo}, respectively, and we will also denote them  by $F$ and $X$. When using  those reduced forms, we will refer not only to the form of the map and the vector field  \eqref{forma_normal_tor} and \eqref{fnormal_tors_edo1} but also to the conditions \eqref{condicions_fnormal}.

For the sake of simplicity  we will only consider the cases of $F$ and $X$ with $\bar a_k(\lambda), \, \bar d_p( \lambda) \neq 0$, which include the generic ones. 
Other more degenerate cases may be treated with the same techniques.
In Theorem \ref{teorema_analitic_helicoure} we consider a case with $\bar a_k(\lambda) = 0$ to be able to deal with an application to the scattering of He atoms off Cu surfaces.

Along the paper we will sometimes omit the dependence of the functions we work with on the parameter $\lambda$ when there is no danger of confusion. Concretely, we present the statements, setting and function spaces with detail but we skip the dependence on parameters in the lemmas and proofs.


To prove the results we use  the parameterization method for invariant manifolds (see
\cite{cfdlll03I}, \cite{cfdlll05}, \cite{HCFLM16}). It consists in looking  for the invariant manifolds of $F$ or $X$ as images of parameterizations together with a representation of the dynamics of $F$ or $X$ restricted to the invariant manifolds.

In the maps setting, we look for a pair of functions, $K(u, \theta, \lambda) : [0, \rho) \times \T^d \times \Lambda \to  \R^2 \times \T^d$ and $R(u, \theta, \lambda) : [0, \rho) \times \T^d \times \Lambda \to  \R \times \T^d$ satisfying the invariance equation
\begin{equation} \label{eq-param-met-def}
	F ( K(u, \theta, \lambda), \lambda) = K (R(u, \theta, \lambda), \lambda).
\end{equation}
This equation establishes that  the range of $K$ is contained in the domain of $F$. It is a functional equation that has to be adapted to the setting of the problem at hand. It follows immediately from \eqref{eq-param-met-def} that the range of $K$ is invariant. Actually, $K$ is a (semi)conjugation of the map $F$ restricted to the range of $K$ to $R$. Then, one has to solve equation \eqref{eq-param-met-def} in a suitable space of functions. Usually it is convenient to have good approximations of $K$ and $R$
and look for a small correction of $K$, in some sense, while maintaining $R$ fixed.

Assuming differentiability and taking derivatives in \eqref{eq-param-met-def} we get 
$DF \circ K \cdot DK  =DK \circ R \cdot DR$ which says that the range of $DK$ has to be invariant by $DF$. Therefore, in our setting we look for $K=( K^x, K^y)$ and $R$ such that $K(0, \theta, \lambda) = (0,0, \theta)$, $R(0, \theta, \lambda) = (0, \theta + \omega)$, and $\partial_u K^y(u,\theta, \lambda) / \partial_u K^x(u,\theta, \lambda) \to 0$ as $u \to 0$.

The existence of $K$ provides the existence of an invariant manifold, stable or unstable depending on whether $0$ is an attractor  or a repellor
for $R$.

In the vector fields setting, to find an invariant manifold of $X$ following the parameterization method, we look for a map $K$ and a vector field $Y$ satisfying 
\begin{equation} \label{eq-camps-conj-1}
	X \circ K = DK \cdot Y.
\end{equation}
This equation expresses that on the range of $K$, the vector field $X$ is tangent to the range of $K$, and therefore,  the image of $K$ is invariant under the flow of $X$. Moreover, the vector field $Y$ is a representation of $X$ restricted to the image of $K$.

When $X$ is nonautonomous one needs a time-dependent version of equation \eqref{eq-camps-conj-1}. Adding  the equation $\dot t =1$ to the system $\dot z = X(z,t)$, $z \in U \times \T^d$, and applying \eqref{eq-camps-conj-1} to the new vector field
we arrive at the equation
\begin{equation} \label{eq-conj-tors-edo}
	X(K(u, \theta, t, \lambda), t, \lambda)- \partial_{(u, \theta)} K (u, \theta, t, \lambda) \cdot Y (u, \theta, t, \lambda) - \partial_t K(u, \theta, t, \lambda) = 0
\end{equation}
for $K$ and $Y$ also depending on $t$. Concretely,  look for a map $K(u, \theta, t, \lambda)$ and a vector field $Y(u,  \theta, t,\lambda)$  satisfying \eqref{eq-conj-tors-edo}.

Then equation \eqref{eq-conj-tors-edo} expresses that on the range of $K$, the vector field $(X, 1)$ is tangent to the range of $K$, and therefore,  the image of $K$ is invariant under the flow of $(X,1)$.  Moreover, we look for $K$ and $Y$ satisfying  $K(0,\theta, t, \lambda) = (0, 0, \theta) \in \R^2 \times \T^d$, $Y(0,\theta, t, \lambda) = (0, \omega) \in \R \times \T^d$ and $\partial_u K^y / \partial_u K^x \to 0$ as $u \to 0$.

It is known that in the parabolic case, in general, there is a loss of regularity of the invariant manifolds around the invariant torus with respect to the regularity of the map or the vector field (\cite{BFM20a, BFM20b, fontich99}). Then we cannot assume \textit{a priori} a Taylor expansion with respect to $u$ of high degree of the manifolds at $u=0$. However, we can obtain formal approximations, $\MK_n$, $\mathcal{R}_n$ and $\MY_n$ of $K$, $R$ and $Y$, satisfying the equations \eqref{eq-param-met-def} and \eqref{eq-conj-tors-edo} up to  any order. In our results we will show that these expressions are indeed approximations of true invariant manifolds, whose existence is rigorously established.  

\subsection{Main results} \label{mainresults}


In this section we state the results on existence of analytic stable  invariant manifolds for maps and vector fields of the form \eqref{forma_normal_tor} and \eqref{fnormal_tors_edo1}, respectively, asymptotic to the invariant torus $\mathcal{T}$. 
For both cases we also state an
\emph{a posteriori} result, which  provides the existence of a stable invariant manifold assuming it has been previously approximated but the statement is independent of the way such an approximation has been obtained. 
In the Appendix at the end of the paper we show that completely analogous results for the unstable manifolds also hold true.

\begin{theorem}[Invariant manifolds of maps] \label{teorema_analitic_tors}
	Let  $F: U \times \T^d \times \Lambda \to \R^2 \times \T^d$ be an analytic map of the form \eqref{forma_normal_tor}. Assume that $2p> k-1$,  $\bar a_k (\lambda) > 0$  for $\lambda \in \Lambda$, and that  $\omega$ is Diophantine.
	Then, there exists $\rho > 0$ and a $C^1$ map  $K: [0, \, \rho ) \times \T^d \times \Lambda \to \R^2 \times \T^d$, analytic in $(0,\rho) \times \T^d \times \Lambda$,  of the form 
	\begin{equation*} \label{propietats_k_analitic_tors}
		K(u, \theta, \lambda) = (u^2, \,  \ol{ K }_{k+1}^y (\lambda) u^{k+1},\,  \theta + \bar K_{2p-k+1}^\theta(\lambda)  u^{2p-k+1} ) + (O(u^3), O(u^{k+2}), O(u^{2p-k+2})),
	\end{equation*}
	and  a polynomial map $R$  of the form 
	$$
	R(u, \theta, \lambda) = (
	u + \bar R^x_k (\lambda) u^k + \ol R^x_{2k-1}(\lambda) u^{2k-1}, \, 
	\theta + \omega ),
	$$ 
	with $\bar R_k^x (\lambda) <0$, such that 
	$$
	F(K(u, \theta, \lambda), \lambda) = K(R(u, \theta, \lambda), \lambda), \qquad (u, \theta, \lambda) \in [0, \, \rho) \times \T^d \times \Lambda.
	$$
	Moreover, we have
	\begin{small}
		$$
		\ol K_{k+1}^y(\lambda) = -  \sqrt{\frac{2 \, \bar a_k(\lambda)}{\bar c(\lambda)\, (k+1)}}, \quad \ol K_{2p-k+1}^\theta (\lambda) = - \frac{\bar d_p(\lambda)}{2p-k+1} \sqrt{\frac{2(k+1)}{\bar c (\lambda) \, \bar a_k}}, \quad \ol R_k^x (\lambda) = - \sqrt{\frac{\bar c(\lambda) \, \bar a_k(\lambda)}{2(k+1)}}.
		$$
	\end{small}
\end{theorem}

\begin{remark} 
	The statement of Theorem \ref{teorema_analitic_tors} provides a local stable manifold parameterized by  $K: [0, \, \rho ) \times \T^d \times \Lambda \to \R^2 \times \T^d$ with $\rho$ small. The proof does not give an explicit estimate for the value of $\rho$, however, one  can extend the domain of $K$ by using the formula 
	$$
	K(u, \theta,  \lambda) = F^{-j} K(R^j(u, \theta, \lambda)), \qquad j \geq 1,
	$$
	while the iterates of the inverse map $F^{-1}$ exist. 
\end{remark}

\begin{theorem}[\emph{A posteriori} result for maps] \label{teorema_posteriori_tors}
	Let $F: U \times \T^d \times \Lambda \to \R^2 \times \T^d$ be an analytic map of the form \eqref{forma_normal_tor} with $ 2p>k-1$, and let  $\wh K:(-\rho, \rho)  \times \T^d \times \Lambda \to \R^2 \times \T^d $ and $\wh R = (-\rho, \rho)  \times \T^d \times \Lambda \to \R \times \T^d$
	be analytic maps of the form
	$$
	\wh K(u, \theta, \lambda) = (u^2, \,  \ol K_{k+1}^y (\lambda) u^{k+1},\,  \theta + \ol K_{2p-k+1}^\theta(\lambda) u^{2p-k+1} ) + (O(u^3), O(u^{k+2}), O(u^{2p-k+2})),
	$$
	and 
	$$\wh R(u, \theta, \lambda) =  
	(	u + \ol R_k^x(\lambda) u^k + O(u^{k+1}) , \, 
	\theta + \omega ),		
	$$
	with  $\ol R_k^x(\lambda)<0$, satisfying
	\begin{equation*} 
		F ( \wh K (u, \theta , \lambda), \lambda) -  \wh K (\wh R(u, \theta, \lambda), \lambda) =  (O(u^{n+k}), O(u^{n+2k-1}), \, O(u^{n+2p-1})),
	\end{equation*}
	for some $n\ge 2$.
	
	Then, there exists a $C^1$ map $K:[0,\rho) \times \T^d \times \Lambda \to \R^2 \times \T^d$, analytic in $(0,\rho) \times \T^d \times \Lambda$, and an analytic map $R:(-\rho,\rho) \times \T^d \times \Lambda \to \R \times \T^d$  such that 
	$$
	F ( K (u , \theta, \lambda), \lambda) =   K (  R(u, \theta , \lambda), \lambda) , \qquad  (u, \theta ) \in [0,\rho) \times \T^d,
	$$
	and 
	$$
	K (u , \theta, \lambda) -  \wh K(u, \theta, \lambda ) =  (O(u^{n+1}), O(u^{n+k}), \, O(u^{n+2p-k})) ,
	$$
	$$ 
	R(u, \theta, \lambda) -\wh R(u, \theta, \lambda) = \left\{ \begin{array}{ll}
		(  O(u^{2k-1}) ,  0 ) & \text{ if }\; n\le k, \\
		(0, 0)  & \text{ if } \; n>k.
	\end{array} \right. 
	$$
\end{theorem}

We also have a uniqueness result for the invariant manifolds obtained in the previous theorems.

\begin{theorem} \label{thm:unicitat}
Under 	
the hypotheses of Theorems \ref{teorema_analitic_tors} and \ref{teorema_posteriori_tors} the stable manifold is unique.
\end{theorem}
	The proof is deferred to the Appendix.
	\vspace{3pt}

	\begin{theorem}[Invariant manifolds of vector fields] \label{teorema_analitic_tors_edo}
		Let $X$ be an analytic vector field of the form  \eqref{fnormal_tors_edo1} and let $\nu \in \R^{d'}$ be the time frequencies of $X$. Assume that $2p> k-1$.  Assume also  that $(\omega, \nu)$ is Diophantine and that  $\bar{a}_k(\lambda) > 0$ for $\lambda \in \Lambda$.
		
		Then, there exists $\rho > 0$ and  a $C^1$ map  $K: [0, \, \rho ) \times \T^d \times \R \times \Lambda \to \R^2 \times \T^d$, analytic in $(0,\rho) \times \T^d \times \R \times \Lambda$,  of the form 
		\begin{equation*}
			K(u, \theta, t, \lambda) = (u^2, \,  \ol K_{k+1}^y(\lambda) u^{k+1},\,  \theta + \ol K_{2p-k+1}^\theta(\lambda) u^{2p-k+1} ) + (O(u^3), O(u^{k+2}), O(u^{2p-k+2})),
		\end{equation*}
		depending quasiperiodically on $t$ with the same frequencies as $X$,
		and a polynomial vector field $Y$  of the form 
		$$
		Y(u, \theta, t, \lambda) = Y(u, \lambda)= 
		(	\ol Y^x_k(\lambda) u^k + \ol Y^x_{2k-1}(\lambda) u^{2k-1}, \, 
			\omega ),
		$$
		with $\ol Y_k^x(\lambda) <0$, such that 
		\begin{align*}
			X(K(u, \theta, t, \lambda), t, \lambda) - \partial_{(u, \theta)} K(u, \theta, t, \lambda)  \cdot Y(u, \theta, t, \lambda) & - \partial_t K(u, \theta, t, \lambda) = 0,\\
			& (u, \theta, t, \lambda) \in [0, \, \rho) \times \T^d \times \R \times \Lambda.
		\end{align*}
	Moreover, we have
	\begin{small}
		$$
		\ol K_{k+1}^y(\lambda) = -  \sqrt{\frac{2 \, \bar a_k(\lambda)}{\bar c(\lambda)\, (k+1)}}, \quad \ol K_{2p-k+1}^\theta (\lambda) = - \frac{\bar d_p(\lambda)}{2p-k+1} \sqrt{\frac{2(k+1)}{\bar c (\lambda) \, \bar a_k}}, \quad \ol Y_k^x (\lambda) = - \sqrt{\frac{\bar c(\lambda) \, \bar a_k(\lambda)}{2(k+1)}}.
		$$
	\end{small}
	\end{theorem}
	
	\vspace{3pt}
	
	\begin{theorem}[\emph{A posteriori} result for vector fields] \label{teorema_posteriori_tors_edo}
		Let $X$ be an analytic vector field of the form  \eqref{fnormal_tors_edo1} and let $\nu \in \R^{d'}$ be the time frequencies of $X$. Let  $\wh K:(-\rho, \rho)  \times \T^d \times \R  \times \Lambda \to \R^2 \times \T^d $ and $\wh Y = (-\rho, \rho)  \times \T^d \times \R \times \Lambda \to \R \times \T^d$
		be an analytic map and an analytic vector field, respectively, of the form
		\begin{equation*}
			K(u, \theta, t, \lambda) = (u^2, \,  \ol K_{k+1}^y(\lambda) u^{k+1},\,  \theta + \ol K_{2p-k+1}^\theta(\lambda) u^{2p-k+1} ) + (O(u^3), O(u^{k+2}), O(u^{2p-k+2})),
		\end{equation*}
		and 
		$$
		Y(u, \theta, t, \lambda) =
			(\ol Y^x_k(\lambda) u^k + O(u^{k+1}), \, 
			\omega ),
		$$
		with  $\ol Y_k^x(\lambda)<0$, depending quasiperiodically on $t$ with the same frequencies as $X$, satisfying
		\begin{align*} 
			X(\wh K(u, \theta, t, \lambda),t, \lambda) - \partial_{(u, \theta)} \wh K(u, \theta, t, \lambda)  \cdot \wh Y(u, \theta, t, \lambda) &- \partial_t \wh K(u, \theta, t, \lambda)  \\
			&=  (O(u^{n+k}), O(u^{n+2k-1}), \, O(u^{n+2p-1})),
		\end{align*}
		for some $n\ge 2$.
		
		Then, there exists a $C^1$ map $K:[0,\rho) \times \T^d \times \R \times \Lambda \to \R^2 \times \T^d$, analytic in $(0,\rho) \times \T^d \times \R \times \Lambda$, and an analytic vector field $Y:(-\rho,\rho) \times \T^d \times \R \times \Lambda \to \R \times \T^d$  such that 
		\begin{align*}
			X(K(u, \theta, t, \lambda), t, \lambda) - \partial_{(u, \theta)} K(u, \theta, t, \lambda)  \cdot Y(u, \theta, t, \lambda) &- \partial_t K(u, \theta, t, \lambda) = 0, \\
			& (u, \theta, t, \lambda) \in [0, \, \rho) \times \T^d \times \R  \times \Lambda,
		\end{align*}
		and 
		$$
		K (u , \theta, t, \lambda) -  \wh K(u, \theta, t, \lambda) =  (O(u^{n+1}), O(u^{n+k}), \, O(u^{n+2p-k})) ,
		$$
		$$ 
		Y(u, \theta, t, \lambda) -\wh Y(u, \theta, t, \lambda) = \left\{ \begin{array}{ll}
			(  O(u^{2k-1}) ,  0 ) & \text{ if }\; n\le k, \\
			(0, 0)  & \text{ if } \; n>k.
		\end{array} \right. 
		$$
	\end{theorem}

	\begin{remark}
	The first component of the map $R$ and the vector field $Y$ (corresponding to the directions normal to the invariant torus) given in Theorems \ref{teorema_analitic_tors} and \ref{teorema_analitic_tors_edo} is the normal form of the dynamics of a one-dimensional system in a neighborhood of a parabolic point (\cite{chen68, takens73}). In the second component, $R$ and $Y$ define a rigid rotation of frequency $\omega$.  
\end{remark}
	
	Finally, the following result is a particular case of a slightly modified version of Theorem \ref{teorema_analitic_tors_edo}. It will be used later on in Section \ref{sec-appl-tors} applied to the study of the scattering of helium atoms off copper surfaces. The proof, which is completely analogous to the one of Theorem \ref{teorema_analitic_tors_edo}, is sketched in Appendix \ref{App-A1}.
	
	\begin{theorem} \label{teorema_analitic_helicoure}
		Let $X$ be an analytic vector field of the form 
		\begin{equation*} 
			X(x, y, \theta) =
			\begin{pmatrix}
				c(\theta) y \\
				b(\theta )x y + O(y^2)  \\
				\omega + d(\theta )y + O(\|(x,y)\|^2)
			\end{pmatrix},
		\end{equation*}
		with $(x, y) \in \R^2$, $\theta \in \T^d$, $\omega \in \R^d$. Assume that $\bar  c > 0$, $\bar b  \neq 0$ and $\bar d \neq 0$. Assume also that $\omega$ is Diophantine.
		
		Then, there exists $\rho > 0$ and  a $C^1$ map  $K: [0, \, \rho ) \times \T^d \to \R^2 \times \T^d$, analytic in $(0,\rho) \times \T^d $,  of the form 
		\begin{equation*}
			K(u, \theta ) = (u, \,  \ol K_{2}^y u^{2},\,  \theta + \ol K_{1}^\theta  u ) + (O(u^2), O(u^{3}), O(u^{2})),
		\end{equation*}
		and a polynomial vector field $Y$  of the form 
		$$
		Y(u, \theta) = Y(u )= 
		(	\ol Y^x_2 u^2 + \ol Y^x_{3} u^{3}, \, 
			\omega ),
		$$ 
		with $\ol Y_2^x <0$, such that 
		\begin{align*}
			X(K(u, \theta  )) - D K(u, \theta)  \cdot Y(u, \theta)  = 0, \qquad (u, \theta) \in [0, \, \rho) \times \T^d.
		\end{align*}
		Moreover, we have
	\begin{equation} \label{coefs1_th_helicoure}
		\ol K_{2}^y =   \frac{b}{2c}, \quad \ol K_{2p-k+1}^\theta = \frac{2d}{c}, \quad \ol Y_2^x  = \frac{b}{2}. 
	\end{equation}	
	\end{theorem}
	
	\subsection{Applications} \label{sec-appl-tors}
	
In this section we present two applications of the previous results. The first one is a simple illustrative example, and the second one  is related to a problem from chemistry.

\subsubsection{A quasiperiodically forced oscillator}

Consider a particle of mass $1$ moving along a straight line under the action of a potential $V(x)$, with $V(x)= c x^{2n}$, $c >0$, $n \in \N$. When $n=1$, the system is a harmonic oscillator. The corresponding equation of motion is  
$$
\ddot x = - V'(x) = -2n c \,  x^{2n-1}.
$$
Denoting $y=\dot x$ the velocity of the particle,  the equations of motion are 
\begin{align}
\begin{split} \label{oscillador_gen}
\dot x & = y, \\
\dot y & = -2nc \, x^{2n-1}.
\end{split}
\end{align}
System \eqref{oscillador_gen} has the first  integral $H(x,y) = \frac{1}{2}  y^2 +  c x^{2n}$, and hence, the  phase space is  foliated by periodic orbits around the origin, corresponding to the closed curves $\frac{1}{2}  y^2 +  c x^{2n}=h$, for energy levels $h >0$.

We show next that perturbing that oscillator with a suitable external force one can break the center character of the origin of the system and introduce a parabolic stable invariant manifold. Assume that the particle moving under the action of the potential $V(x)$ is also submitted to an external analytic force, $F$ that may depend  on the position $x$, the velocity $\dot x$ and the time $t$.  
 
Now the equations of motion become
\begin{align}
\begin{split} \label{qp_forcat_1dim}
\dot x & = y, \\
\dot y & = -2nc \, x^{2n-1} + F(x, y, t).
\end{split}
\end{align}
System \eqref{qp_forcat_1dim} is a particular case of system \eqref{fnormal_tors_edo1} without its third component. Moreover, if the analytic function $f(x, y, t) = -2ncx^{2n-1} + F(x,y,t)$ satisfies the hypotheses \eqref{condicions_fnormal} and the dependence on $t$ is Diophantine with frequencies $\nu \in \R^d$, then system \eqref{qp_forcat_1dim} satisfies the hypotheses of Theorem \ref{teorema_analitic_tors_edo}.

As a concrete example, take $n=2$ and
$F(x, y, t) = \alpha x^2 g(t)$, with $\alpha >0$ and $g$ a quasiperiodic function of $t$ with frequency $\nu \in \R^d$, $\nu$ Diophantine, and $\bar g >0$. This system is modeled by 
\begin{align}
\begin{split} \label{oscillador1}
& \dot x = y, \\
& \dot y =  - 4cx^3  + \alpha  x^2 g(t).
\end{split}
\end{align}
By Theorem \ref{teorema_analitic_tors_edo}  there are solutions of system \eqref{oscillador1} asymptotic to $(0,0)$, analytic away from $(0,0)$, contained in the stable manifold of the origin. Moreover, one can apply Proposition \ref{prop_tors_edo_sim} (see Section \ref{sec-algoritme-tors}) to obtain the coefficients of an approximation of a parameterization of such stable manifold. 

To look for an unstable manifold of system \eqref{oscillador1} we consider the vector field obtained after changing the sign of the time, $t \mapsto -t$, in \eqref{oscillador1}. The stable manifold of the transformed  system, namely
\begin{align}
\begin{split} \label{unstable-qp-1}
& \dot x = -y, \\
& \dot y =  4cx^3  - \alpha  x^2 g(-t),
\end{split}
\end{align}
will be unstable manifold of system \eqref{oscillador1}. Performing the change of variables $y \mapsto -y$ to system \eqref{unstable-qp-1} we can apply again Theorem \ref{teorema_analitic_tors_edo} to obtain the existence of an analytic stable manifold. Finally, undoing the  changes of variables we have that, system \eqref{oscillador1} has a stable manifold in the lower right plane and an unstable manifold in the upper right plane, both of them analytic away from the origin.  
We have then that for every $ \alpha >0$, an external force  of the form $\alpha  x^2 g(t)$ with  the conditions stated before breaks the oscillatory behavior of the system and induces  solutions that brings the particle to the origin and others out of the it.

\subsubsection{Scattering of He atoms off Cu corrugated surfaces}

In the paper\cite{gbm97}, the authors study the phase-space structure of a differential equation modelling the scattering of helium atoms off copper corrugated surfaces.
	Concretely,  elastic collisions of $^4$He atoms with corrugated Cu surfaces are considered, in particular those made of Cu(110) and Cu(117). The system, which can be adequately treated at the classical level,  can be modeled by the following two degrees of freedom Hamiltonian describing the motion of a $^4$He  atom, 
	\begin{equation} \label{ham-he-cu}
		H(x, z, p_x, p_z) = \frac{p_x^2 + p_z^2}{2m} + V(x, z), 
	\end{equation}
	where $x$ is the coordinate parallel to the copper surface and $z$ is the coordinate perpendicular to it, $p_x$ and $p_z$ are the respective momenta, and $m$ is the mass of the atom. The potential energy $V(x, z)$ is given by
	$$
	V(x, z) = V_M(z) + V_C(x, z),
	$$
	where $V_M(z) = D(1-e^{-\alpha z})^2$ is the Morse potential and 
	$V_C(x, z) = D e^{-2 \alpha z} g (x)$ is the coupling potential, with $D=6.35$ meV, $\alpha = 1.05$ $\AA^{-1}$, and $g(x)$ is a periodic function. Thus the variable $x$ can be thought as an angle. For more information on the coefficients of the Morse and coupling potentials, see Table 1 of \cite{gbm97}.
	
	The equations of motion derived from the Hamiltonian function  \eqref{ham-he-cu} are 
	\begin{equation} \label{eq-ham-he-cu}
		\dot x = \frac{p_x}{m}, \qquad
		\dot z =  \frac{p_z}{m}, \qquad
		\dot p_x = -D  e^{-2 \alpha z} g'(x), \qquad
		\dot p_z = -2D\alpha e^{-\alpha z} + 2 D \alpha e^{-2\alpha z} (1+g(x)).
	\end{equation}
This system has a periodic orbir at infinity (see below for the precise meaning). 
	
We will use the results presented in Section \ref{mainresults} to show that the parabolic periodic orbit at infinity has stable and unstable invariant manifolds. 
This means that for certain initial conditions the helium atom goes away spiraling asymptotically to a periodic orbit, and also, for other initial conditions with position far away, the atom comes asymptotically from the periodic orbit.

	Since \eqref{eq-ham-he-cu} is a Hamiltonian system, the energy $H$ is conserved, and thus each  solution of the system is contained in a level set $H(x, z, p_x, p_z) = h$. Therefore we can reduce system \eqref{eq-ham-he-cu} to a three dimensional system restricting it to an energy level, $H(x, z, p_x, p_z) = h$,  removing the equation for $\dot p_x$. The obtained system reads
	\begin{align*}
		\dot x &= \frac{1}{m} \big{(} 2m(h-D(1- e^{-\alpha z})^2 - D e^{-2 \alpha z} g(x)) - p_z^2 \big{)}^{1/2}, \\
		\dot z &= \frac{p_z}{m}, \\ 
		\dot p_z &= -2D\alpha e^{-\alpha z} + 2 D \alpha e^{-2\alpha z} (1+g(x)).
	\end{align*}
	Next, to study the motion for very big values of $z$ we perform the McGehee-like change of variables given by $y = -e^{-\alpha z}$. Now the set $y=0$ corresponds to  infinit distance from the copper surface. To adapt the notation to the one of Section \ref{mainresults}, we write $ \theta= x$ and $p=p_z$. We get
	\begin{align} \begin{split} \label{sist-fisic-he-cu}
			\dot p &= 2 D \alpha y + 2 D \alpha y^2 (1+g(\theta)), \\
			\dot y &= - \frac{\alpha}{m} py, \\
			\dot \theta &= \frac{1}{m} \big{(} 2m(h-D(1-y)^2 - D y^2 g(\theta)) - p^2 \big{)}^{1/2},
		\end{split}
	\end{align}
	with $y \leq 0$, $p \in \R$ and $\theta \in  \T$ 
	
The set $\{p = y = 0\} $ is invariant for \eqref{sist-fisic-he-cu}, and corresponds to a periodic orbit at the infinity of system \eqref{eq-ham-he-cu}. 
	 For system \eqref{sist-fisic-he-cu} we have the following result.

	\begin{theorem} 
		Let $X$ be the vector field associated to system  \eqref{sist-fisic-he-cu}, and assume that $h > D$. Then,  the set $\gamma = \{ p = y = 0 \}$ is a periodic orbit and it has stable and unstable invariant manifolds.
		Concretely, there exist $\rho > 0$ and two $C^1$ maps, $K^-, K^+: [0,\rho) \times \T \to \R^2 \times \T$, analytic on $(0,\rho) \times \T$, and two analytic vector fields, $Y^-, Y^+:  [0,\rho) \times \T \to \R \times \T$ 
		of the form
		\begin{equation} \label{param-sta-hecu}
			K^-(u, \theta) = \begin{pmatrix}
				u + O(u^2) \\
				K_2^y u^2 + O(u^3) \\
				K_1^\theta u + O(u^2)
			\end{pmatrix}, \qquad 
			Y^-(u, \theta) = \begin{pmatrix}
				-Y_2 u^2 +  Y_3 u^3 \\
				\omega
			\end{pmatrix},
		\end{equation}
		corresponding to the stable manifold, and 
		\begin{equation} \label{param-inest-hecu}
			K^+(u, \theta) = \begin{pmatrix}
				u + O(u^2) \\
				-K_2^y u^2 + O(u^3) \\
				K_1^\theta u + O(u^2)
			\end{pmatrix}, \qquad 
			Y^+(u, \theta) = \begin{pmatrix}
				Y_2 u^2 + \wh Y_3 u^3 \\
				\omega
			\end{pmatrix},
		\end{equation}
		corresponding to the unstable manifold, with
		\begin{equation} \label{coefs-sta-hecu}
			K_2^y = - \frac{1}{4mD} , \qquad    K_1^\theta = -\frac{1}{\alpha \sqrt{2m(h-D)}}, \qquad
			Y_2 = \frac{\alpha}{2m} , \qquad   \omega =  \sqrt{\frac{2(h-D)}{m}},
		\end{equation}
		such that 
		$$
		X \circ K^- (u,\theta) = DK^- \cdot Y^- (u, \theta) \ \ \text{ and } \ \ X \circ K^+ (u,\theta) = DK^+ \cdot Y^+ (u, \theta), \qquad (u, \theta) \in [0, \rho) \times \T.
		$$
\end{theorem}
	\begin{proof}
			We do the following analytic change of variables  to system \eqref{sist-fisic-he-cu},
		\begin{equation} \label{canvi-var-hecu}
			\tilde p = p, \qquad \tilde y =  y + (1+g(\theta))y^2, \qquad  \tilde \theta = \theta,  
		\end{equation}
		and we expand the right hand side of the third equation in Taylor series around $(p, y)=(0,0)$, so that the new system reads, writing the new variables without tilde, 
		\begin{equation} \label{sist-he-cu-adaptat}
			\dot p = 2D\alpha y, \qquad \dot y = -\frac{\alpha}{m} p y + O(y^2), \qquad \dot \theta = \omega - d_1 y + O(\|(p, y)\|^2),
		\end{equation}
		with $d_1 = \frac{D}{\sqrt{2m (h-D)}}$.

		It is clear that system \eqref{sist-he-cu-adaptat} has a periodic orbit, $\gamma$, at $p=y=0$ parameterized by $\gamma (t) = (0,0, \omega t)$. Moreover, such system satisfies the hypotheses of Theorem \ref{teorema_analitic_helicoure} with $d' =0$, $c(\theta, \lambda) = 2 D \alpha$, $b(\theta, \lambda) = - \tfrac{\alpha}{m}$, and $d(\theta, \lambda) = -d_1$.

	Then, the stated results are a direct consequence of Theorem \ref{teorema_analitic_helicoure}, which provides the existence of an analytic stable invariant manifold of system \eqref{sist-he-cu-adaptat} and the  expressions given in  \eqref{param-sta-hecu} and \eqref{coefs-sta-hecu}.

		Undoing the change of variables \eqref{canvi-var-hecu} we obtain the parameterizations, $K^-$ and $Y^-$, of the stable manifold of $\gamma$ for the original system and the restricted dynamics on it, whose lower order coefficients are the ones in \eqref{coefs-sta-hecu}.

		The existence of the unstable manifold is obtained through the application of Theorem \ref{teorema_analitic_helicoure} to the system obtained doing the change $t \mapsto -t$ to \eqref{sist-he-cu-adaptat}.
		Also, performing the change  $p \mapsto -p$, we obtain 
		\begin{equation} \label{sist-he-cu-3}
			\dot p = 2D\alpha y, \quad \dot y = -\frac{\alpha}{m} p y + O(y^2), \quad \dot \theta = -\omega + d_1 y + O(\|(p, y)\|^2).
		\end{equation}
		Then we can apply again Theorem \ref{teorema_analitic_helicoure} to system \eqref{sist-he-cu-3}, which provides an analytic stable invariant manifold, $\wt K^+$,  asymptotic to $\gamma = \{p=y=0\}$, and an expression for the restricted dynamics, $	\wt Y^+$, parameterized by
		\begin{equation*}
			\wt K^+(u, \theta) = \begin{pmatrix}
				u + O(u^2) \\
				K_2^y u^2 + O(u^3) \\
				-K_1^\theta u + O(u^2)
			\end{pmatrix}, \qquad 
			\wt Y^+(u, \theta) = \begin{pmatrix}
				-Y_2 u^2 + \wt Y_3 u^3 \\
				-\omega
			\end{pmatrix}.
		\end{equation*}
		Finally, going back to the original variables we get the parameterizations of the unstable manifold of $\gamma$ and the restricted dynamics on it, namely $K^+$ and $Y^+$, given in \eqref{param-inest-hecu}.
	\end{proof}

	\section{Formal approximation of  parameterizations of the manifolds} \label{sec-algoritme-tors}

				In this section we provide an algorithm to compute an approximation of a parameterization of the invariant manifolds of a map of the form \eqref{forma_normal_tor} and a vector field  of the form \eqref{fnormal_tors_edo1}.
				
			From now on, the superindices $x$, $y$ and $\theta$ on the symbol of a function or an operator with values in $\R^2 \times \T^d$ denote the respective components of the function or the operator. In the next sections we also use the superindices $u, \,  \theta$ and $t$, respectively, for functions or operators that take values in $\C \times \T^d_\sigma \times \T^{d'}_\sigma$.

	\subsection{The case of maps} \label{sec-algoritme-tors-maps}
	
First, we recall a basic result concerning Diophantine vectors and the small divisors equation.
	
	Let 
	$$
	\T^d_\sigma = \{   \theta = (\theta_1 , \cdots , \theta_d)  \in (\C/\Z)^d \,  | \ |\text{Im} \,  \theta| < \sigma  \}
	$$
	denote a complex torus of dimension $d$. We also denote  by $\Lambda_\C$ a complex neighborhood of the parameter space, $\Lambda$.

		We say that a vector $\omega \in \R^d$ is \emph{Diophantine (in the map setting)} if there exists $c >0$ and $\tau \geq d$ such that 
		$$
		|\omega \cdot k - l| \geq c|k|^{-\tau}  \qquad \text{for all} \quad k \in \Z^d \backslash \{0\}, \, l \in \Z,
		$$ 
		where $|k| = |k_1| + \dots + |k_d|$ and $\omega \cdot k$ denotes the scalar product.
		
		Along this section, when solving cohomological equations, we will encounter the so-called \emph{small divisors equation.} In the map setting this equation has the following form, 
		\begin{equation} \label{SDequation}
			\varphi (\theta + \omega, \lambda) - \varphi(\theta, \lambda) = h (\theta, \lambda),
		\end{equation}
		with $h : \T^d \times  \Lambda \to \R^n$ and $\omega \in \R^d$. 
		To find a solution $\varphi(\theta, \lambda)$ of \eqref{SDequation} we consider the Fourier expansion of $h$ with respect to $\theta$: 
		$$
		h(\theta, \lambda) = \sum_{k \in \Z^d} h_k(\lambda) e^{2 \pi i k \cdot \theta} .
		$$
		If $h$ has zero average and $k \cdot \omega \notin \Z$ for all $k \neq 0$, then equation \eqref{SDequation} has the formal solution 
		$$
		\varphi (\theta, \lambda ) = \sum_{k \in \Z^d} \varphi_k(\lambda) e^{2 \pi i k \cdot \theta}, \qquad \varphi_k(\lambda)  = \frac{h_k(\lambda)}{1- e^{2\pi i k \cdot \omega}}, \qquad k \neq 0.
		$$
		Note that all coefficients $\varphi_k$ are uniquely determined except $\varphi_0$ (the average of $\varphi$), which is free. 
		
		The following well-known result establishes the existence of a solution to equation \eqref{SDequation} when $h$ is analytic.
		\begin{lemma}[Small divisors lemma for maps] \label{SDlemma_maps}
			Let $h: \T^d_\sigma \times \Lambda_\C \to \C^n$ be an analytic function such  that $\sup_{(\theta, \lambda) \in \T^d_{\sigma}\times \Lambda_\C} \|h(\theta, \lambda)\| \, < \infty$ and having zero average.  Let $\omega \in \R^d$ be Diophantine with $\tau  \geq d$. Then, there exists a unique analytic solution $\varphi : \T^d_\sigma \times \Lambda_\C \to \C^n$ of \eqref{SDequation} with zero average. Moreover,
			$$
			\sup_{(\theta, \lambda ) \in \T^d_{\sigma-\delta} \times \Lambda_\C} \|\varphi(\theta, \lambda)\| \, \leq \, C \delta^{-\tau} \sup_{(\theta, \lambda) \in \T^d_\sigma \times \Lambda_\C} \,  \|h(\theta, \lambda)\|, \qquad 0 < \delta < \sigma,
			$$
			where $C$ depends on $\tau$ and $d$ but not on $\delta$.
		\end{lemma}
		
		The proof with close to optimal estimates is due to Russmann \cite{Rus75}. 	We will denote by $\mathcal{SD}(h)$ the unique solution of \eqref{SDequation} with zero average.  
		
		In the next result we  obtain two pairs of maps, $\mathcal{K}_n$ and $\mathcal{R}_n$, that  are approximations of solutions $K$ and $R$ of the invariance equation 
		\begin{equation*}
			F \circ K = K \circ R.
		\end{equation*}
	The obtained approximations correspond to the stable manifold when the coefficient $\ol R_k^x(\lambda)$  of $\mathcal{R}_n$ is negative, and to the unstable manifold when such coefficient is positive.	Moreover, the obtained parameterizations of $\MK_n$ and $\MR_n$ will satisfy the hypotheses of Theorem \ref{teorema_posteriori_tors}, and therefore $\MK_n$ will be an approximation of a true invariant manifold of $F$. Moreover, the first component of $\MR_n$ coincides with the expression of the normal form of a one-dimensional map around a parabolic point (\cite{chen68, takens73}).

	\begin{theorem}[A computable appoximation for maps] \label{prop_simple_tors}
		Let $F$ be an analytic map of the form \eqref{forma_normal_tor} satisfying the hypotheses \eqref{condicions_fnormal}. Assume that $2p > k-1$, $\bar{c}(\lambda), \bar{a}_k(\lambda) > 0$ for $\lambda \in \Lambda$, and that  $\omega$ is Diophantine. Then, for all $n \geq 2$, there exist two pairs of  maps,  $\MK_n : \R \times \T^d \times \Lambda \to  \R^2 \times \T^d$ and $\mathcal{R}_n:  \R \times \T^d \times \Lambda \to  \R \times \T^d$, of the form
		$$ 
		\MK_{n} (u, \theta, \lambda)= 
		\begin{pmatrix}
			u^2 +  \sum_{i=3}^n \ol K_i^x(\lambda)  u^i + \sum_{i=k+1}^{n+k-1} \wt{ K}_{i}^x(\theta, \lambda)u^{i}  \\
			\sum_{i=k+1}^{n+k-1} \ol K_{i}^y(\lambda) u^{i} + \sum_{i=2k}^{n+2k-2}   \wt{ K}_{i}^y(\theta, \lambda)u^{i} \\
			\theta + \sum_{i=2p-k+1}^{n+2p-k-1} \ol K_{i}^\theta(\lambda) u^{i} + \sum_{i=2p}^{n+2p-2}  \wt{ K}_{i}^\theta(\theta, \lambda)u^{i} 
		\end{pmatrix}
		$$
		and 
		$$ \mathcal{R}_n(u, \theta, \lambda) =
		\begin{cases}
			\begin{pmatrix}
				u+ \ol R_k^x(\lambda) u^k \\
				\theta + \omega 
			\end{pmatrix}	 &\qquad \text{ if }  \;\; 2 \leq n \leq k, \\
			\begin{pmatrix}
				u+ \ol R_k^x(\lambda) u^k +  \ol R_{2k-1}^x(\lambda) u^{2k-1}  \\
				\theta + \omega 
			\end{pmatrix}
			& \qquad \text{ if } \;\; n\ge k+1, 
		\end{cases}
		$$
		such that 
		\begin{equation} \label{eq_prop1_tor_sim}
			\MG_n (u, \theta, \lambda) := F(\MK_{n}(u, \theta, \lambda), \lambda)- \MK_{n}(\mathcal{R}_n(u, \theta, \lambda)\lambda) = (O(u^{n+k}),\, O(u^{n+2k-1}), \, O(u^{n+2p-1})).
		\end{equation}
		Moreover, for the lowest order coefficients we have
		\begin{small}
			\begin{align}
				\label{coefcasmapes}
				& 	\ol K_{k+1}^y(\lambda) = \pm  \sqrt{\frac{2 \, \bar a_k(\lambda)}{\bar c(\lambda)\, (k+1)}}, \ \ \ol K_{2p-k+1}^\theta (\lambda) = \pm \frac{\bar d_p(\lambda)}{2p-k+1} \sqrt{\frac{2(k+1)}{\bar c (\lambda) \,  \bar a_k(\lambda)}}, \ \ \ol R_k^x (\lambda) = \pm \sqrt{\frac{\bar c(\lambda) \, \bar a_k(\lambda)}{2(k+1)}},\\
				\nonumber
				& \wt{ K}_{k+1}^x (\theta, \lambda ) = \mathcal{SD}(\tilde{c}(\theta, \lambda)) \ol{K}_{k+1}^y(\lambda), \quad  \wt{ K}_{2k}^y (\theta, \lambda ) =\mathcal{SD}(\tilde{a}_k(\theta, \lambda)), \quad  \wt{ K}_{2p}^\theta (\theta , \lambda) =\mathcal{SD}(\tilde{d}_p(\theta, \lambda)).	
			\end{align}
		\end{small}
	\end{theorem}
	
	\vspace{1pt}
\begin{remark} \label{rem:KnRnsonpolis}
Although $\MK_{n} $ and $\mathcal R_n$ are polynomials in $u$ and therefore are defined for all $u\in \R $, we only consider them for $u\ge 0$, so that choosing the sign $-$ in   \eqref{coefcasmapes} we get an approximation of the stable manifold  and choosing the sign $+$ we get the unstable one.
\end{remark}
	
	\begin{notation} Along the proof, given a  map $f(u, \theta)$, we will denote by $[f]_n$ the coefficient of the term of order $n$ of the jet of $f$ with respect to $u$ at $0$ . 
	\end{notation}

\vspace{-12pt}

		\begin{proof} We prove the result by induction on  $n$ showing that we can determine $\MK_n$ and $\mathcal{R}_n$ iteratively. 
		
		For the first induction step, $n=2$, we claim that there exist maps of the form
		$$
		\MK_{2}(u, \theta) =
		\begin{pmatrix}
			u^2 + \wt{ K}_{k+1}^x(\theta) u^{k+1} \\
			\ol{K}_{k+1}^y u^{k+1} + \wt{ K}_{2k}^y(\theta) u^{2k} \\
			\theta + \ol K_{2p-k+1}^\theta u^{2p-k+1} + \wt{ K}^\theta_{2p} (\theta) u^{2p}
		\end{pmatrix}, \qquad \MR_2(t, \theta) = 
		\begin{pmatrix}
			u + \ol R_k^x u^k \\
			\theta + \omega 
		\end{pmatrix},
		$$
		such that $\MG_2 (u, \theta) = F(\MK_{2}(u, \theta))- \MK_{2}(\mathcal{R}_2(u, \theta)) = (O(u^{k+2}),  O(u^{2k+1}), O(u^{2p+1}))$.
		
		Indeed, from the expansion of $\MG_2$, since $2p > k-1$ we have
		\begin{align*}
			& \MG_2^x (u, \theta) = u^{k+1}[\wt{ K}_{k+1}^x(\theta) - \wt{ K}_{k+1}^x(\theta + \omega) + c(\theta)\ol{K}_{k+1}^y - 2 \ol{R}_k^x] + O(u^{k+2}), \\
			& \MG_2^y (u, \theta) = u^{2k}[\wt{ K}_{2k}^y (\theta) - \wt{ K}_{2k}^y (\theta + \omega) +  a_k(\theta) - (k+1) \ol{K}_{k+1}^y \ol{R}_k^x] + O(u^{2k+1}), \\
			& \MG_2^\theta (u, \theta) = u^{2p}[\wt{ K}_{2p}^\theta (\theta) -\wt{ K}_{2p}^\theta (\theta + \omega) + d_p(\theta) -  (2p-k+1) 	\ol K_{2p-k+1}^\theta \ol R_k^x ] + O(u^{2p+1}). 
		\end{align*}
		To obtain $\MG_2^x (u, \theta) = O(u^{k+2})$ we have to solve the equation
		$$ \wt{ K}_{k+1}^x(\theta) - \wt{ K}_{k+1}^x(\theta + \omega) + c(\theta)\ol{K}_{k+1}^y - 2 \ol{R}_k^x = 0.
		$$ 
		We proceed as follows. We separate the average and the oscillatory part of the functions that depend on $\theta$ 
		and we split the equation into two parts, one containing the terms that are independent of $\theta$, namely $ \bar c \ol{K}_{k+1}^y =  2 \ol{R}_k^x$, and the other being a small divisors equation of functions with zero average, $  \wt{ K}_{k+1}^x(\theta + \omega) - \wt{ K}_{k+1}^x(\theta) =  \tilde{ c}(\theta)\ol{K}_{k+1}^y$.
		
		We proceed in the same way to get $\MG_2^y (u, \theta) = O(u^{2k+1})$ and $\MG_2^\theta (u, \theta) = O(u^{2p+1})$. Since $\bar{c}, \bar{a}_k >0$ and $\omega$ is Diophantine, the obtained equations have the solutions $\ol K_{k+1}^y, \, \ol K_{2p-k+1}^\theta, \,  \ol R_k^x$, $\wt{ K}_{k+1}^x (\theta ), \, \wt{K}_{2k}^y (\theta )$ and $\wt{ K}_{2p}^\theta (\theta )$ given in the statement.

		Next we perform the induction procedure. We assume that we have already obtained maps $\MK_n$ and $\mathcal{R}_n$, $n \geq 2$, such that \eqref{eq_prop1_tor_sim} holds true, and we look for 
		\begin{align*}
		\MK_{n+1} (u, \theta) &= \MK_n(u, \theta) + \begin{pmatrix} \ol K_{n+1}^x u^{n+1} + \wt{ K}_{n+k}^x(\theta) u^{n+k} \\ \ol K_{n+k}^y \, u^{n+k} + \wt{ K}_{n+2k-1}^y(\theta) u^{n+2k-1}  \\
				\ol K_{n+2p-k}^\theta u^{n+2p-k}	+ \wt{ K}_{n+2p-1}^\theta(\theta) u^{n+2p-1}
			\end{pmatrix},  \\
			\mathcal{R}_{n+1}(u, \theta) &= \mathcal{R}_n(u, \theta) +
			\begin{pmatrix}
				\ol R^x_{n+k-1} \, u^{n+k-1} \\
				0
			\end{pmatrix},
		\end{align*}
		such that  $\MG_{n+1}(u, \theta) = (O(u^{n+k+1}), \, O(u^{n+2k}), \, O(u^{n+2p}))$. To simplify the notation, we denote $\MK_{n+1}^+ = \MK_{n+1} - \MK_n$ and $\MR_{n+1}^+ = \MR_{n+1} - \MR_n$.

		Using Taylor's theorem, we write
		\begin{align*}
			\MG_{n+1}(u, \theta)  &= F(\MK_n(u, \theta) + \MK_{n+1}^+ (u, \theta) )  - (\MK_n(u, \theta ) + \MK_{n+1}^+ (u, \theta))\circ (\mathcal{R}_n(u, \theta) + \MR_{n+1}^+(u, \theta))  \\
			& = \MG_n(u, \theta) + DF(K_n(u, \theta )) \cdot \MK_{n+1}^+ (u, \theta) - \MK_{n+1}^+ (u, \theta) \circ (\mathcal{R}_n(u, \theta) + \MR_{n+1}^+ (u, \theta)  ) \\
			& \quad+ \int_0^1 \hspace{-1pt} (1-s)  D^2F(\MK_n(u, \theta) + s \, \MK_{n+1}^+ (u, \theta) ) \, ds \, \MK_{n+1}^+ (u, \theta)^{\otimes 2}  \\
			&	\quad  - D\MK_n \circ \mathcal{R}_n(u, \theta ) \cdot  \MR_{n+1}^+(u, \theta)  \\
			& \quad - \int_0^1 \hspace{-1pt} (1-s)  D^2\MK_n (\mathcal{R}_n(u, \theta) + s  \, \MR_{n+1}^+(u, \theta) ) \, ds \, \MR_{n+1}^+(u, \theta)^{\otimes 2}.
		\end{align*}
Expanding the components of the previous expression we have
		\begin{align}
			\begin{split} \label{eqs_pasn_alg_sim}
				& \MG_{n+1}^x (u, \theta)  = \MG_n^x (u, \theta) \\
				& \quad + u^{n+k}[\wt{ K}_{n+k}^x(\theta) - \wt{ K}_{n+k}^x(\theta + \omega) + c(\theta)\ol{K}_{n+k}^y - (n+1) \ol K_{n+1}^x \ol{R}_k^x  - 2 \ol R_{n+k-1}^x] + O(u^{n+k+1}), \\
				& \MG_{n+1}^y (u, \theta)  = \MG_n^y (u, \theta)  \\
				&\quad + u^{n+2k-1}[\wt{ K}_{n+2k-1}^y (\theta) - \wt{ K}_{n+2-1}^y (\theta + \omega) +  k \, a_k(\theta) \ol K_{n+1}^x - (n+k) \ol{K}_{n+k}^y \ol{R}_k^x \\
				& \quad - (k+1) \ol K_{k+1}^y \ol R_{n+k-1}^x  ] + O(u^{n+2k}), \\
				& \MG_{n+1}^\theta (u, \theta)  = \MG_n^\theta (u, \theta) \\
				& \quad + u^{n+2p-1}[\wt{ K}_{n+2p-1}^\theta (\theta) -\wt{ K}_{n+2p-1}^\theta (\theta + \omega) + p \, d_p(\theta) \ol K_{n+1}^x \\
				& \quad - (n+2p-k) \ol K_{n+2p-k}^\theta \ol R_k^x - (2p-k+1) \ol K_{2p-k+1}^\theta \ol R_{n+k-1}] + O(u^{n+2p}).
			\end{split} 
		\end{align}
		Since, by the induction hypothesis,
		$\MG_{n}(u, \theta) = (O(u^{n+k}), \, O(u^{n+2k-1}), \, O(u^{n+2p-1}))$,  to complete the induction step we need to make $[\MG_{n+1}^x]_{n+k}$, $[\MG_{n+1}^y]_{n+2k-1}$ and $[\MG_{n+1}^\theta]_{n+2p-1}$  vanish. From  \eqref{eqs_pasn_alg_sim}, such conditions lead to the following cohomological equations,  
		\begin{align}
			\begin{split} \label{6eqs_cohom_n_sim}
				&  \wt{ K}_{n+k}^x(\theta) - \wt{ K}_{n+k}^x(\theta + \omega) + c(\theta)\ol{K}_{n+k}^y - (n+1) \ol K_{n+1}^x \ol{R}_k^x  - 2 \ol R_{n+k-1}^x + [\MG_{n}^x(\theta)]_{n+k} =0, \\
				&   \wt{ K}_{n+2k-1}^y (\theta) - \wt{ K}_{n+2-1}^y (\theta + \omega) +  k \, a_k(\theta) \ol K_{n+1}^x - (n+k) \ol{K}_{n+k}^y \ol{R}_k^x  \\
				& \qquad \qquad \qquad \qquad \qquad \qquad \qquad \qquad \qquad \ \ \, \qquad   - (k+1) \ol K_{k+1}^y \ol R_{n+k-1}^x  + [\MG_n^y(\theta)]_{n+2k-1} = 0, \\
				&  \wt{ K}_{n+2p-1}^\theta (\theta) -\wt{ K}_{n+2p-1}^\theta (\theta + \omega) + p \, d_p(\theta) \ol K_{n+1}^x  \\
				&  \qquad \qquad \qquad - (n+2p-k) \ol K_{n+2p-k}^\theta \ol R_k^x  - (2p-k+1) \ol K_{2p-k+1}^\theta \ol R_{n+k-1}+ [\MG_n^\theta (\theta)]_{n+2p-1} = 0.
			\end{split}
		\end{align}
Taking averages with respect to $\theta$ in the previous equations and separating the terms that depend on $\theta$ from the constant ones, we split \eqref{6eqs_cohom_n_sim}  into three small divisors equations of functions with zero average, namely, 
		\begin{align} \label{SDequations_induccio_sim}
			\begin{split}	
				&   \wt{ K}_{n+k}^x(\theta + \omega) - \wt{ K}_{n+k}^x(\theta) =  \tilde c(\theta)  \ol K_{n+k}^y + [\wt\MG_{n}^x(\theta)]_{n+k}, \\
				&   \wt{ K}_{n+2-1k}^y (\theta + \omega) - \wt{ K}_{n+2k-1}^y (\theta) =   k \, \tilde a_k(\theta) \ol K_{n+1}^x  + [\wt \MG_n^y(\theta)]_{n+2k-1}, \\
				&  \wt{ K}_{n+2p-1}^\theta (\theta + \omega) - \wt{ K}_{n+2p-1}^\theta (\theta) =   p \, \tilde d_p(\theta) \ol K_{n+1}^x + [\wt \MG_n^\theta (\theta)]_{n+2p-1},
			\end{split}
		\end{align}
		and the following linear system of equations with constant coefficients,  
		\begin{align}
			\begin{split}
				\label{sist_lineal_tors_sim}
				& \begin{pmatrix}
					-(n+1) \ol R_k^x & \bar c & 0  \\
					k \, \bar a_k & -(n+k) \, \ol R_k^x  & 0 \\
					p\, \bar d_p & 0 & -(n+2p-k) \ol R_k^x
				\end{pmatrix} \hspace{-2pt}
				\begin{pmatrix}
					\ol K_{n+1}^x \\
					\ol K_{n+k}^y \\
					\ol K_{n+2p-k}^\theta
				\end{pmatrix} \\
				& \qquad = 
				\begin{pmatrix}
					- [\ol \MG_n^x]_{n+k} + 2 \ol R_{n+k-1}^x \\
					- [\ol \MG_n^y]_{n+2k-1}  + (k+1)  \ol K_{k+1}^y \ol R_{n+k-1}^x \\
					- [\ol \MG_n^\theta ]_{n+2p-1} + (2p-k+1) \ol K_{2p-k+1}^\theta \ol R_{n+k-1}^x
				\end{pmatrix}.
			\end{split}
		\end{align}	
Note that the determinant of the matrix in the left hand side of \eqref{sist_lineal_tors_sim} is 
$$
 (n+2p-k) \ol R_k^x \,[  k  \, \bar c  \, \bar a_k -  (n+1)(n+k) (\ol R_k^x)^2] 
 = (n+2p-k) \ol R_k^x \,  \, \bar c  \, \bar a_k
 \frac{(k-n)(n+2k+1)}{2(k+1)},
$$ 
		 which vanishes when $k  \, \bar c  \, \bar a_k - (n+1)(n+k) (\ol R_k^x)^2 = 0$.  Since by hypothesis $n+2p-k \geq 1$, if $n \neq k$ the matrix is invertible, so we can take $\ol R_{n+k-1}^x = 0$ and  then obtain $\ol K_{n+1}^x$, $\ol K_{n+k}^y$ and $\ol  K_{n+2p-k}^\theta$ in a unique way. When $n=k$ the determinant of the matrix is zero. Then, choosing 
		$$
		\ol R_{2k-1}^x =  \frac{2k \, \ol R_k^x \, [\ol \MG_n^x]_{2k} + \bar c \, [\MG_n^y]_{3k-1}}{2 \, (3k+1) \, \ol R_k^x},
		$$
		system \eqref{sist_lineal_tors_sim} has solutions. In this case, however, $\ol K_{k+1}^x$,  $\ol K_{2k}^y$ and $\ol  K_{2p}^\theta$ are not uniquely determined.
		
		Once we have chosen solutions $\ol K_{n+1}^x$,  $\ol K_{n+k}^y$ and $\ol  K_{n+2p-k}^\theta$ of system \eqref{sist_lineal_tors_sim}, we solve the small divisors equations in  \eqref{SDequations_induccio_sim} taking 
		\begin{align*}
			& \wt{ K}_{n+k}^x(\theta) = \mathcal{SD} (  \tilde c(\theta) \ol K_{n+k}^y + [\wt \MG_{n}^x(\theta)]_{n+k}), \\
			& \wt{ K}_{n+2k-1}^y (\theta)  = \mathcal{SD} (  k \, \tilde a_k(\theta) \ol K_{n+1}^x  + [\wt \MG_n^y(\theta)]_{n+2k-1}), \\
			& \wt{ K}_{n+2p-1}^\theta (\theta)   = \mathcal{SD}  (  p \, \tilde d_p(\theta) \ol K_{n+1}^x + [\wt \MG_n^\theta (\theta)]_{n+2p-1}).
		\end{align*}
		In this way all equations in \eqref{6eqs_cohom_n_sim} are solved and one can  proceed to the next induction step. 
	\end{proof}

\subsection{The case of vector fields} \label{sec_algoritme_tors_edo}
	
In this case we have an analogous results. 
We denote $\mathbb{H}_\sigma = \{ z \in \C \, | \ |\text{Im}(z)|  <  \sigma\}$ the complex strip of thickness $2\sigma >0$.

	We say that $\omega \in \R^d$ is \emph{Diophantine (in the vector field setting)}  if there exist $c > 0$ and $\tau \geq d-1$ such that
	$$
	|\omega \cdot k| \geq c |k|^{-\tau} \qquad \text{for all} \qquad k \in \Z^d \backslash \{ 0 \}.
	$$
	The small divisors equation in the vector field setting is 
	\begin{equation} \label{SDequation_edo}
		\partial_\theta \varphi(\theta, \lambda) \cdot \omega = h(\theta, \lambda),
	\end{equation}
	with $h : \T^d \times \Lambda \to \R^n$ and $\omega \in \R^d$. 
	
Similarly to the case of maps, if $h$ has zero average and $k \cdot \omega \notin \Z$ for all $k \neq 0$, then equation \eqref{SDequation_edo} has the formal solution 
	$$
	\varphi (\theta, \lambda ) = \sum_{k \in \Z^d} \varphi_k (\lambda) e^{2 \pi i k \cdot \theta}, \qquad \varphi_k (\lambda)  = \frac{h_k (\lambda)}{2 \pi i k \cdot \omega}, \qquad k \neq 0,
	$$
	where all the coefficients $\varphi_k$ are uniquely determined except $\varphi_0$ which is free.

	\begin{lemma}[Small divisors lemma for vector fields] \label{SDlemma_edo}
		Let $h: \T^d_\sigma \times \Lambda_\C \to \C^n$ be an analytic function such  that $\sup_{(\theta, \lambda) \in \T^d_{\sigma}\times \Lambda_\C} \|h(\theta, \lambda)\| \, < \infty$ and having zero average. Let $\omega$ be Diophantine with $\tau  \geq d -1$. Then, there exists a unique analytic solution $\varphi : \T^d_\sigma \times \Lambda_\C \to \C^n$ of \eqref{SDequation_edo} with zero average. Moreover, 
		$$
		\sup_{(\theta, \lambda) \in \T^d_{\sigma-\delta}\times \Lambda_\C} \|\varphi(\theta, \lambda)\| \, \leq \, C \delta^{-\tau} \sup_{(\theta, \lambda) \in \T^d_\sigma \times \Lambda_\C} \,  \|h(\theta, \lambda)\|, \qquad 0 < \delta < \sigma,
		$$
		where $C$ depends on $\tau$ and $d$ but not on $\delta$.
	\end{lemma}
	
	As in the case of maps we denote by $\mathcal{SD}(h)$ the unique solution of \eqref{SDequation_edo} with zero average.

	As a consequence, if $h: \T^d_\sigma \times \MAH_\sigma \times \Lambda_\C \to \C^n$ is quasiperiodic with respect to  $t\in \MAH$ with frequencies $\nu \in \R^{d'}$, $h$ has zero average and $(\omega, \nu) \in \R^{d + d'}$ is Diophantine, then the equation
	\begin{equation} \label{SDeq_time}
		(\partial_\theta \varphi(\theta, t, \lambda) , \partial_t \varphi(\theta, t, \lambda))\cdot (\omega, 1) = h(\theta, t, \lambda), \qquad (\theta, t, \lambda) \in \T^d_\sigma \times \MAH_\sigma \times \Lambda_\C,
	\end{equation}
	has a unique solution with zero average, defined in $\T^d_\sigma \times \MAH_\sigma \times \Lambda_\C$ and bounded in $\T^d_{\sigma'} \times \MAH_{\sigma'} \times \Lambda_\C$ for any $0 < \sigma' < \sigma$. Indeed, since $h$ is quasiperiodic in $t$, equation \eqref{SDeq_time} is equivalent to 
	\begin{equation} \label{SD_time_hat}
		(\partial_\theta \check \varphi(\theta, \tau, \lambda), \partial_\tau \check \varphi (\theta, \tau, \lambda)) \cdot (\omega, \nu) = \check h(\theta, \tau, \lambda), \qquad (\theta, \tau, \lambda) \in \T^{d+d'}_\sigma \times \Lambda_\C,
	\end{equation}
	where $\tau = \nu t$ and $h(\theta, t, \lambda) = \check h (\theta, \tau,\lambda)$. Then, applying Theorem \ref{SDlemma_edo} to equation \eqref{SD_time_hat} taking $(\omega, \nu)$ as the frequency vector, we obtain a unique solution $\check \varphi : \T^{d+d'}_{\sigma'}\times \Lambda_\C \to \C^n$ with zero average, and thus $\varphi(\theta, t, \lambda) = \check \varphi(\theta, \tau, \lambda)$ is the unique solution of equation \eqref{SDeq_time} with zero average. We also denote it by $\mathcal{SD}(h)$. We use the same notation to denote the solution of a small divisors equation that is either time dependent or independent, as such dependence will be understood by the context.  
	
	In the next result, given an analytic vector field $X$ of the form \eqref{fnormal_tors_edo1}, we obtain two maps, $\mathcal{K}_n (u, \theta, t, \lambda)$ and two vector fields,  $\mathcal{Y}_n (u, \theta, t, \lambda)$, that  are approximations of solutions $K$ and $Y$ of the invariance equation 
	\begin{equation} \label{conj_tors_edo}
		X \circ (K,t) - \partial_{(u, \theta )} K \cdot Y - \partial_t K = 0,
	\end{equation}
Note that the obtained vector field $\MY_n$ neither depends on $\theta$ nor on $t$. Moreover, the first component of $\MY_n$, which represents the dynamics in a transversal directions to the invariant torus, coincides with the expression of the normal form of a one-dimensional vector field around a parabolic point given in (\cite{takens73}).
	
	\begin{theorem}[A computable approximation for vector fields] \label{prop_tors_edo_sim}
		Let $X$ be an analytic vector field of the form \eqref{fnormal_tors_edo1} satisying the hypotheses \eqref{condicions_fnormal}.  Assume that $2p> k-1$. Assume also that $(\omega, \nu)$ is Diophantine and  $\bar{a}_k (\lambda) > 0$ for $\lambda \in \Lambda$. Then, for all $n \geq 2$, there exist two maps,  $\MK_n : \R \times \T^d \times \R  \times \Lambda \to  \R^2 \times \T^d$, of the form 
		$$ 
		\MK_{n} (u, \theta, t,  \lambda)= 
		\begin{pmatrix}
			u^2 +  \sum_{i=3}^n \ol K_i^x( \lambda)  u^i + \sum_{i=k+1}^{n+k-1} \wt{ K}_{i}^x(\theta, t, \lambda)u^{i}  \\
			\sum_{i=k+1}^{n+k-1} \ol K_{i}^y(\lambda) u^{i} + \sum_{i=2k}^{n+2k-2}   \wt{ K}_{i}^y(\theta, t, \lambda)u^{i} \\
			\theta + \sum_{i=2p-k+1}^{n+2p-k-1} \ol K_{i}^\theta(\lambda) u^{i} + \sum_{i=2p}^{n+2p-2}  \wt{ K}_{i}^\theta(\theta, t, \lambda)u^{i} 
		\end{pmatrix},
		$$
		depending quasiperiodically on time with the same frequencies as $X$, 
		and two vector fields,  $\mathcal{Y}_n:  \R \times \T^d \times \R \times \Lambda \to  \R \times \T^d$, of the form
		$$ \mathcal{Y}_n(u, \theta, t, \lambda) = \mathcal{Y}_n(u, \lambda)
		= \begin{cases}
			\begin{pmatrix}
				\ol Y_k^x(\lambda) u^k \\
				\omega 
			\end{pmatrix}	 &\qquad \text{ if }  \;\; 2 \leq n \leq k, \\
			\begin{pmatrix}
				\ol Y_k^x(\lambda) u^k + \ol Y_{2k-1}^x(\lambda) u^{2k-1}  \\
				\omega 
			\end{pmatrix}
			& \qquad \text{ if } \;\; n\ge k+1, 
		\end{cases}
		$$
		such that 
		\begin{align}
			\begin{split} \label{eq_prop1_tor_edo_sim}
				\MG_n (u,  \theta, t, \lambda) &:= X(\MK_{n}(u, \theta, t, \lambda), t, \lambda)- \partial_{(u, \theta)}\MK_{n} (u, \theta, t, \lambda) \cdot \mathcal{Y}_n(u, \theta, t, \lambda)  - \partial_t \MK_n(u, \theta, t, \lambda)\\
				& \; = (O(u^{n+k}),\, O(u^{n+2k-1}), \, O(u^{n+2p-1})).
			\end{split}
		\end{align}
		Moreover, for the lowest order coefficients we obtain
		\begin{small}
			\begin{align*}
				& 	\ol K_{k+1}^y(\lambda) = \pm  \sqrt{\frac{2 \, \bar a_k(\lambda)}{\bar c(\lambda)\, (k+1)}}, \qquad \ol K_{2p-k+1}^\theta(\lambda) = \pm \frac{\bar d_p(\lambda)}{2p-k+1} \sqrt{\frac{2(k+1)}{\bar c(\lambda) \, \bar a_k(\lambda)}}, 
				\qquad \ol Y_k^x(\lambda) = \pm \sqrt{\frac{\bar c(\lambda)\, \bar a_k(\lambda)}{2(k+1)}} ,  \\
				& \wt{ K}_{k+1}^x (\theta, t, \lambda ) = \mathcal{SD}(\tilde{ c}(\theta,t, \lambda)) \ol{K}_{k+1}^y(\lambda), \qquad  \wt{ K}_{2k}^y (\theta, t, \lambda ) =\mathcal{SD}(\tilde{a}_k(\theta,t, \lambda )), \qquad \wt{ K}_{2p}^\theta (\theta, t, \lambda ) =\mathcal{SD}(\tilde{d}_p(\theta,t, \lambda)).	
			\end{align*}
		\end{small}
		\vspace{-18pt}	
	\end{theorem}
A remark analogous to Remark \ref{rem:KnRnsonpolis} also applies here.

	\begin{proof} The proof is analogous to the one of Theorem \ref{prop_simple_tors}, but in this case we look for parameterizations $\MK_n$ and $\MY_n$ that approximate solutions of equation \eqref{conj_tors_edo}.
		
		For the first induction step, $n=2$, we claim that there exist a map and a vector field,
		$$
		\MK_{2}(u, \theta, t) =
		\begin{pmatrix}
			u^2 + \wt{ K}_{k+1}^x(\theta, t) u^{k+1} \\
			\ol{K}_{k+1}^y u^{k+1} + \wt{ K}_{2k}^y(\theta, t) u^{2k} \\
			\theta +  \ol K_{2p-k+1}^\theta u^{2p-k+1} + \wt{ K}^\theta_{2p} (\theta, t) u^{2p}
		\end{pmatrix}, \qquad \MY_2(u, \theta, t) = 
		\begin{pmatrix}
			\ol Y_k^x u^k \\
			\omega 
		\end{pmatrix},
		$$
		such that 
		\begin{align} \begin{split} \label{expans_vf}
			\MG_2 (u, \theta, t) &= X(\MK_{2}(u, \theta, t), t)- \partial_{(u, \theta)}\MK_{2}(u, \theta, t) \cdot \mathcal{Y}_2(u, \theta, t) - \partial_t \MK_2 (u, \theta, t) \\
			&= (O(u^{k+2}),  O(u^{2k+1}), O(u^{2p+1})).
	\end{split}	\end{align}
This leads to a set of $d+2$ cohomological equations, that we split into two parts, one containing the terms that are independent of $(\theta, t)$,  and the other being a small divisors equation for functions of $(\theta, t)$ with zero average.
	
	Then,  since $\bar{c},\bar{a}_k >0$ and $(\omega, \nu)$ is Diophantine, by the small divisors lemma the equations obtained from \eqref{expans_vf} have solutions $	\ol K_{k+1}^y, \, \ol K_{2p-k+1}^\theta , \, \ol Y_k^x$, $\wt{ K}_{k+1}^x (\theta, t ), \, \wt{ K}_{2k}^y (\theta, t )$ and $\wt{ K}_{2p}^\theta (\theta, t )$ as given in the statement. We emphasize that we obtain two solutions. The next terms will depend on the choice we make for  those solutions.

		In the induction procedure we look for 
		\begin{align*}
			\MK_{n+1} (u, \theta, t) &= \MK_n(u, \theta, t) + \begin{pmatrix} \ol K_{n+1}^x u^{n+1} + \wt{ K}_{n+k}^x(\theta, t) u^{n+k} \\ \ol K_{n+k}^y \, u^{n+k} + \wt{ K}_{n+2k-1}^y(\theta, t) u^{n+2k-1}  \\
				\ol K_{n+2p-k}^\theta u^{n+2p-k}	\wt{ K}_{n+2p-1}^\theta(\theta, t) u^{n+2p-1}
			\end{pmatrix},  \\
			\mathcal{Y}_{n+1}(u, \theta, t) &= \mathcal{Y}_n(u, \theta, t) +
			\begin{pmatrix}
				\ol Y^x_{n+k-1} \, u^{n+k-1} \\
				0
			\end{pmatrix},
		\end{align*}
		such that  $\MG_{n+1}(u, \theta,t) = (O(u^{n+k+1}), \, O(u^{n+2k}), \, O(u^{n+2p}))$. 
		
	Proceeding in the same way as in the case of maps we arrive to the following completely analogous equations for the average and the oscillatory parts of the coefficients of $\MK_n$ and $\MY_n$,
		\begin{align} \label{SDequations_induccio_edo_sim}
			\begin{split}	
				& \partial_\theta  \wt{ K}_{n+k}^x(\theta , t) \cdot \omega + \partial_t \wt{ K}_{n+k}^x(\theta, t) =  \tilde c(\theta, t)  \ol K_{n+k}^y + [\wt \MG_{n}^x(\theta, t)]_{n+k}, \\
				&  \partial_\theta \wt{ K}_{n+2-1k}^y (\theta , t)\cdot \omega + \partial_t \wt{ K}_{n+2k-1}^y (\theta, t) =   k \, \tilde a_k(\theta, t) \ol K_{n+1}^x  + [\wt \MG_n^y(\theta, t)]_{n+2k-1}, \\
				&  \partial_\theta \wt{ K}_{n+2p-1}^\theta (\theta , t)\cdot \omega +\partial_t \wt{ K}_{n+2p-1}^\theta (\theta, t) =   p \, \tilde d_p(\theta, t) \ol K_{n+1}^x + [\wt \MG_n^\theta (\theta, t)]_{n+2p-1},
			\end{split}
		\end{align}
		and 
		\begin{align}
			\begin{split}
				\label{sist_lineal_tors_edo_sim}
				& \begin{pmatrix}
					-(n+1) \ol Y_k^x & \bar c & 0  \\
					k \, \bar a_k & -(n+k) \, \ol Y_k^x  & 0 \\
					p\, \bar d_p & 0 & - (n+2p-k) \ol Y_k^x
				\end{pmatrix} \hspace{-2pt}
				\begin{pmatrix}
					\ol K_{n+1}^x \\
					\ol K_{n+k}^y \\
					\ol K_{n+2p-k}^\theta
				\end{pmatrix} \\
				& \qquad  = 
				\begin{pmatrix}
					- [\ol \MG_n^x]_{n+k} + 2 \ol Y_{n+k-1}^x \\
					- [\ol \MG_n^y]_{n+2k-1}  + (k+1)  K_{k+1}^y \ol Y_{n+k-1}^x \\
					- [\ol \MG_n^\theta ]_{n+2p-1} + (2p-k+1) \ol K_{2p-k+1}^\theta \ol Y_{n+k-1}^x
				\end{pmatrix}.
			\end{split}
		\end{align}
		As in Theorem \ref{prop_simple_tors}, the matrix in the left hand side of \eqref{sist_lineal_tors_edo_sim} is invertible provided that $n \neq k$. In such case one can take $\ol Y_{n+k-1}^x = 0$ and we obtain $\ol K_{n+1}^x$, $\ol K_{n+k}^y$ and $\ol K_{n+2p-k}^\theta$ in a unique way. When $n=k$, the determinant of the matrix is zero. Choosing
		$$
		\ol Y_{2k-1}^x =  \frac{2k \, \ol Y_k^x \, [\ol \MG_n^x]_{2k} + \bar c \, [\MG_n^y]_{3k-2}}{2 \, (3k+1) \, \ol Y_k^x},
		$$
		system \eqref{sist_lineal_tors_edo_sim} has solutions. In this case,  $\ol K_{k+1}^x$,  $\ol K_{2k}^y$ and $\ol K_{2p}^\theta$ are not uniquely determined.
		
		Once we have chosen solutions $\ol K_{k+1}^x$,  $\ol K_{2k}^y$ and $\ol K_{n+2p-k}^\theta$ for system \eqref{sist_lineal_tors_edo_sim}  we proceed as in the case of maps.  
	\end{proof}
	
	\section{A functional equation for a parametrization of the stable manifold} \label{sec-funcional}
	
	In this section we explain the approach to study the existence of  stable invariant manifolds for analytic maps of the form \eqref{forma_normal_tor} and analytic time-dependent vector fields of the form \eqref{fnormal_tors_edo1}. We establish a functional equation for a parametrization of the stable invariant manifolds and we present the function spaces and operators that we will use. The treatment  in the map and the vector field settings are somehow analogous, so we will omit some details in the latter. 
	
	\subsection{The case of maps} \label{sec-funcional-maps}

	To study the existence of a stable invariant manifold of a map of the form \eqref{forma_normal_tor},  we first consider approximations  $\MK_n : \R \times \T^d \times \Lambda \to \R^2 \times \T^d$ and $\mathcal{R}_n: \R \times \T^d \times \Lambda \to \R \times \T^d$   of solutions of the equation 
	\begin{equation} \label{funcional-tors-1}
		F \circ K = K \circ R,
	\end{equation}
	obtained in Section \ref{sec-algoritme-tors-maps} up to a high enough order, to be determined later on. Then, keeping $R=\mathcal{R}_n$ fixed, we look for a correction $\Delta: [0, \, \rho)  \times \T^d \times \Lambda \to \R^2 \times \T^d$, for some $\rho>0$, of $\MK_n$, analytic on $(0,\rho) \times \T^d \times  \Lambda$, such that the pair $K= \MK_n + \Delta $, $R=\mathcal{R}_n$  satisfies the invariance condition
	\begin{equation} \label{eq_delta_analitic_tors}
		F \circ (\MK_n + \Delta) - (\MK_n + \Delta) \circ R = 0.
	\end{equation}
	The proof of Theorems \ref{teorema_analitic_tors} and \ref{teorema_posteriori_tors} concerning the stable manifolds is organized as follows. First, we rewrite equation \eqref{eq_delta_analitic_tors}  to separate the dominant linear part with respect to $\Delta$ and the remaining terms. This motivates the introduction of two families of operators, $\MS_{n, \, R}^\times$ and $\MN_{n, \, F}$, and the spaces where these operators will act on. We provide the properties of these operators in Lemmas \ref{invers_S_tor}  and \ref{lema_N_tors}, in particular the invertibility of $\MS^\times_{n, R}$.
	Finally, we rewrite the equation for $\Delta $ as the fixed point equation
	$$
	\Delta = \MT_{n,\, F} (\Delta), \qquad \text{where}\qquad \MT_{n, \, F} = (\MS_{n, \, R}^\times)^{-1} \circ \MN_{n, \, F},
	$$
	and we apply the Banach fixed point theorem to get the solution. The needed properties of the operators $\MT_{n, F}$ are given in  Lemma \ref{lema_contraccio_tors}.  
		
	Let $F: U \times \T^d \times \Lambda \to \R^2  \times \T^d$ be an analytic map of the form \eqref{forma_normal_tor}:
	\begin{equation*} \label{forma_simplificada_tor} 
		F(x, y, \theta, \lambda) = 
		\begin{pmatrix}
			x + c(\theta, \lambda) y \\
			y + P(x, y , \theta, \lambda) \\
			\theta + \omega + Q(x, y, \theta, \lambda)
		\end{pmatrix},
	\end{equation*}
	where $P(x, y, \theta, \lambda) = a_k (\theta, \lambda)x^k + A(x, y, \theta, \lambda)$ and $Q(x, y, \theta, \lambda) = d_p (\theta, \lambda)x^p + B(x, y, \theta, \lambda)$ and $A$ and $B$ have the form \eqref{condicions_fnormal}.
	
	From Proposition \ref{prop_simple_tors}, given $n \geq 2$ there exist  $\MK_n$ and $R= \mathcal{R}_n$,  polynomial in $u$, such that 
	\begin{equation} \label{analitic_hipotesi_tors}
		F \circ \MK_n - \MK_n \circ R = \MG_n, 
	\end{equation}
	where $\MG_n(t, \theta) = (O(t^{n+k}), \,O(t^{n+2k-1}), \, O(t^{n+2p-1}))$.
	Since we are looking for a stable manifold of $F$ we will take the approximation corresponding to $R=\mathcal{R}_n$ with the coefficient $\ol R_k^x(\lambda) <0$.
	
	Hence, we look for $\rho>0$ and a map $ \Delta: [0, \, \rho)  \times \T^d \times \Lambda \to \R^2 \times \T^d$, 
	$\Delta = (\Delta^x, \Delta^y, \Delta^\theta) = (O(u^n), \, O(u^{n+k-1}), O(u^{n+2p-k-1}))$, satisfying \eqref{eq_delta_analitic_tors},
	where $\MK_n$ and $R$ are the mentioned maps that satisfy \eqref{analitic_hipotesi_tors}.

	Using \eqref{analitic_hipotesi_tors} we  can rewrite \eqref{eq_delta_analitic_tors} as 
	\begin{align}  \begin{split}\label{eqdelta_analitic_tors}
			\Delta^x \circ R - \Delta^x  &= \MK_n^y [c\circ (\MK_n^\theta + \Delta^\theta) - c \circ \MK_n^\theta ]  + \Delta^y \, c\circ(\MK_n^\theta + \Delta ^\theta) +  \MG_n^x ,  \\
			\Delta^y \circ R -  \Delta^y  & = P \circ (\MK_n + \Delta) - P \circ \MK_n + \MG_n^y ,\\
			\Delta^\theta \circ R - \Delta^\theta &  = Q \circ (\MK_n + \Delta) - Q \circ \MK_n + \MG_n^\theta. 
		\end{split}
	\end{align}
Given $\rho \in (0, 1)$ and $\beta \in (0, \frac{\pi}{k-1})$,   let $S$ be the complex sector
	$$
	S = S(\beta, \rho) = \big\{ z \in \C \, | \ |\arg(z)| < \tfrac{\beta}{2}, \, 0 < |z| < \rho \big\}.
	$$

	\begin{definition}
		Given a sector $S = S(\beta, \, \rho)$, the complex torus $\T^d_{\sigma} $ with $\sigma > 0$, $ \Lambda_{\C} \subset \C^\ell$ and $n \in \N$, let $\MW_n$ be the Banach space 
		$$
		\MW_{n} = \bigg{\{} f : S \times \T^d_{\sigma} \times \Lambda_{\C} \rightarrow \C \, | \ f \text{ real analytic, }    \; \|f\|_n: = \sup_{(u, \theta, \lambda) \in S \times \T^d_{\sigma} \times \Lambda_{\C} } \frac{|f(u,\theta, \lambda)|}{|u|^n} < \infty \bigg{\}},
		$$
		with the norm $\|\cdot\|_n$.
	\end{definition}

	Note that when $n\ge 1$ the functions $f$ in $\MW_n$ can be continuously extended to $z=0$ with $f(0, \theta, \lambda )= 0$ and, if moreover we have  $n\ge 2$, the derivative of $f$ with respect to $z$ can be continuously extended  to  $z=0$ with $\tfrac{\partial f}{\partial u}(0, \theta, \lambda)=0$.
	
	Note also that $\MW_{n+1} \subset \MW_n$ for all $n \in  \mathbb{N}$,
	and that if  $f\in \MW_{n+1} $, then $\|f\|_n \le \|f\|_{n+1} $. More concretely we have that  $\|f\|_n \le \rho \|f\|_{n+1}$.
	Moreover, if $f\in \MW_m, \, g \in  \MW_n$, then $f g \in \MW_{m+n}$ and $\|fg\|_{m+n} \leq \|f\|_m \,\|g\|_n.$

	Given a product  space, $\prod_i \MW_i$, we endow it with the product norm 
	$$
	\|f\|_{\prod_i \MW_i} = \max_i { \|f_i \|_{\MW_i}},
	$$
	where $f_i = \pi_i \circ f$, and $\pi_i$ is the canonical projection from $\prod_j \MW_j$ to $\MW_i$.
	
	Next  we define the spaces
	\begin{equation*}
		\MW_n^\times = \MW_n \times \MW_{n+k-1} \times \MW_{n+2p-k-1}^d,
	\end{equation*}
	endowed with the product norm defined above. Note that, in our setting,  the functions in $\MW_{n+2p-k-1} $ are mapped into $\C / \Z$. 
	
	We will use the notation $\mathcal{B}_{\alpha}$ to denote a closed ball of radius $\alpha$ not always belonging to the same space. Such space will be understood by the context. For instance, we will write
	\begin{align*}
		\mathcal{B}_{\alpha}  &= \{ f  = (f^x, f^y, f^\theta)\in \MW_n^\times \, | \   \|f\|_{\MW_n^\times} \leq \alpha   \} \subset \MW_n^\times.
	\end{align*} 
		For the sake of simplicity, we will omit the parameters $\rho$, $\beta$ and $\sigma$  in the notation of the spaces $\MW_n$ and $\MW_n^\times$. We will consider  $\Lambda$ bounded. If not, we will work locally in bounded subsets of $\Lambda$.

	Since $F$ is analytic in $U \times \T^d \times \Lambda$, which is relatively compact, it has a holomorphic extension to some complex neighborhood of the form $U_\C \times \T^d_{\sigma} \times \Lambda_\C$ that contains $U \times \T^d \times \Lambda$, where $U_\C$ is a neighborhood of $(0,0)$ in $\C^2$, $\T^d_{\sigma}$ is a complex $d$-dimensional torus and $\Lambda_\C$ is a complex extension of $\Lambda$. Moreover since $\MK_n$ and $R$ are analytic maps, their domain extends to a complex domain of the form $S(\beta, \, \rho) \times \T^d_{\sigma'} \times \Lambda_\C$.

	Then it is possible to set equation \eqref{eqdelta_analitic_tors} in a space of holomorphic functions defined on $S(\beta, \, \rho) \times \T^d_{\sigma'} \times \Lambda_\C$, and to look for $\Delta$ being a real analytic function of  complex variables.
	To solve equation \eqref{eqdelta_analitic_tors}, we will consider $n $ big enough and we will look for a solution, $\Delta \in \mathcal{B}_{\alpha} \subset \MW_n^\times$, for some $\alpha >0$. 
	In what follows, we describe some conditions on $\alpha$.
	
	For compositions in \eqref{eqdelta_analitic_tors} to make sense, we need to ensure that the range of $\MK_n+ \Delta $  is contained in the domain where $F$ is analytic. 		
	Let $b > 0$ be the radius of a closed ball in $\C^2$ contained in $U_\C$, and let $\tilde \sigma < \sigma$.
	We have to consider $\rho >0$ and $ \Delta$ such that $((\MK_n+\Delta)^x, (\MK_n+\Delta)^y) \in U_\C$, $(\MK_n + \Delta)^\theta \in \T^d_{ \sigma}$. To this end we will ensure  that 
	\begin{equation} \label{restrict_1}
		|((\MK_n+\Delta)^x, (\MK_n+\Delta)^y)| \leq b \ \ \ \text{ and } \ \ \  | \text{Im} ((\MK_n + \Delta)^\theta)| \leq \tilde \sigma.
	\end{equation}
We choose $\rho$ and $\sigma'$ small enough such that
	$\sup_{S(\beta, \rho) \times \T^d_{\sigma'} \times \Lambda_\C} |(\MK_n^x (u, \theta,\lambda), \MK_n^y(u, \theta, \lambda))| \leq \tfrac{b}{2}$ and such that $\sup_{S(\beta, \rho) \times \T^d_{\sigma'} \times \Lambda_\C}  | \text{Im}(\MK_n^\theta (u, \theta, \lambda))| \leq \tfrac{\tilde \sigma}{2}$.  Later on we may take smaller values of $\rho$.

	We choose
	\begin{equation} \label{def_alfa}
	\alpha = \min \,  \big\{  \tfrac{1}{2}, \, \tfrac{b}{2}, \,  \tfrac{\tilde \sigma}{2} \big\}.
	\end{equation}
	Therefore, for $\Delta \in \mathcal{B}_\alpha \subset \MW_n^\times$,
	$$
	\sup_{S(\beta, \rho) \times \T^d_{\sigma '} \times \Lambda_\C} |\Delta^x (u, \theta, \lambda)| \leq \sup_{S} \|\Delta^x\|_{n} \, |u|^n \leq \alpha \, \rho^n \leq \tfrac{b}{2} \, \rho^n,  
	$$
	and similarly, $\sup_{S \times \T^d_{\sigma'} \times \Lambda_\C} |\Delta^y (u, \theta, \lambda)| \leq \tfrac{b}{2} \, \rho^{n+k-1}$, and 
	$$
	\sup_{S \times \T^d_{\sigma'} \times \Lambda_\C} |\Delta^\theta (u, \theta, \lambda)| \leq \sup_{S} \|\Delta^\theta\|_{n+2p-k-1} \, |u|^{n+2p-k-1} \leq \alpha \, \rho^{n+2p-k-1} \leq \tfrac{\tilde \sigma}{2} \, \rho^{n+2p-k-1},
	$$
	and in particular, $ |\text{Im}(\Delta^\theta) |  \leq \tilde \sigma/2$. Hence, with this choice of $\alpha$ the condition \eqref{restrict_1} holds true.

	Below we introduce two families of operators that will be used to deal with  \eqref{eqdelta_analitic_tors}. The definition of such operators is motivated by the equation itself. 
	
	We recall the next lemma, Lemma 2.4.2 from \cite{BFM17} that we state here with a slightly modified notation adapted to  the one of this paper.
	\begin{lemma} \label{lema_sector_tor}
		Let $R^x: S(\beta, \rho) \rightarrow \C$ be a holomorphic function of the form $R^x(u) = u +R_k u^k + O(|u|^{k+1})$, with $R_k<0$ and $k \geq 2$. Assume that $0<\beta < \frac{\pi}{k-1}$. Then, for any $\mu \in (0,\,   (k-1)|R_k| \cos \kappa)$, with $\kappa =   \frac{k-1}{2}\beta $, there exists $\rho>0$ small enough such that $R^x$ maps $S(\beta, \rho)$ into itself and
		$$
		| (R^x)^j(u)| \leq \frac{|u|}{(1+ j \, \mu \, |u|^{k-1})^{1/k-1}},  \qquad  u \in S(\beta, \, \rho),\quad  j \geq 0,
		$$
		where  $ (R^x)^j$ refers to the $j$-th iterate of the map $ R^x$.
	\end{lemma}

	\begin{definition} \label{def_S_tor}
		Let $n \geq 1$, $\beta \in (0, \tfrac{\pi}{k-1}) $, and let 
		$
		R : S (\beta, \rho) \times \T^d_{\sigma' } \to S (\beta, \rho) \times \T^d_{\sigma' }
		$
		be an analytic map of the form 
		\begin{equation} \label{R_def_Smaps}
			R(u, \theta)  = \begin{pmatrix}
				u +  R_k u^k + O(u^{k+1}) \\
				\theta + \omega 
			\end{pmatrix},
		\end{equation}
		where $R_k<0$ and the terms $O(u^{k+1})$  do not depend on $\theta$.
		
		We define 
		$\MS_{n, \, R}: \MW_n \rightarrow  \MW_n$, as the linear operator given by
		\begin{align*}
			\MS_{n, \, R}\, f =   f \circ R  - f.
		\end{align*} 
	\end{definition}
	
	\begin{remark} \label{remark_composicio_R}
		By Lemma \ref{lema_sector_tor}, for a map $R$ as in Definition \ref{def_S_tor}, we have that $R^x (u, \theta) = R^x(u)$  maps $S(\beta, \rho)$ into itself, and also, since  $\omega \in \R^d$,  $R^\theta (u, \theta) = R^\theta (\theta) $ maps $ \T^d_{\sigma'}$ into itself. Moreover,  the functions $f \in \MW_n$ are defined on $S(\beta, \rho) \times \T^d_{\sigma'}$, and thus the composition in the definition of $\MS_{n, R}$ is well defined.
		\end{remark}

	The following lemma states  that  the operators $\MS_{n,\,  R}$ have a bounded right inverse and provides a bound  for $\|\MS_{n, \, R}^{-1}\|$.  It is a slightly modified version of Lemma 5.6 of \cite{CC-EF-20}. Its proof will be omitted. 
	
	\begin{lemma} \label{invers_S_tor}
		The operator $\MS_{n, \, R} : \MW_n \to \MW_n $ with $n \geq 1$ has a bounded right inverse, 
		$$
		\MS_{n, \, R}^{-1} : \MW_{n+k-1} \to \MW_n
		$$
		given by 
		\begin{equation} \label{sol_parabolic_tor}
			\MS_{n, \, R}^{  -1} \, \eta = - \, \sum_{j=0}^\infty \, \eta \circ R^j,  \qquad \eta \in \MW_{n+k-1}.
		\end{equation}
		Moreover, for any fixed $\mu \in (0, \,\,  (k-1) |R_k^x | \cos \kappa)$, with $\kappa = \frac{k-1}{2}\beta $, there exists $\rho >0$  such that, taking $S(\beta, \, \rho) \times \T^d_{\sigma'}$ as the domain of the functions  of $\MW_{n+k-1}$,  we have the bound
		$$
		\|(\MS_{n, \, R})^{-1}\| \leq \rho^{k-1}  + \tfrac{1}{\mu} \, \tfrac{k-1}{n}.
		$$
	\end{lemma} 

		\begin{definition} \label{def_N_analitic_tor}
		Let $F$ be the holomorphic extension of an analytic map of the form \eqref{forma_normal_tor} satisfying the hypotheses of Theorem \ref{teorema_analitic_tors}. Let $\alpha$ be as in \eqref{def_alfa}.
		
		Given $n \geq 3$   we introduce $\MN_{n, \, F} = (\MN^x_{n,  F},\, \MN^y_{n,  F}, \, \MN_{n, F}^\theta): \mathcal{B}_{\alpha} \subset \MW_n^\times \to \MW_{n+k-1}^\times$, given by 
		\begin{align*}
		\MN_{n, \, F}^x (f) &= \MK_n^y [c\circ (\MK_n^\theta + f^\theta) - c \circ \MK_n^\theta ]  + f^y \, c\circ(\MK_n^\theta + f^\theta) +  \MG_n^x ,  \\
		\MN_{n, \, F}^y(f) & = P \circ (\MK_n + f) - P \circ \MK_n + \MG_n^y ,\\
		\MN_{n, \, F}^\theta(f)  &  = Q \circ (\MK_n + f) - Q \circ \MK_n + \MG_n^\theta.
		\end{align*}
	\end{definition}
	
	In the following lemma we show that the operators $\MN_{n, \, F}$ are Lipschitz  and we provide  bounds for their Lipschitz constants. 
	
	\begin{lemma} \label{lema_N_tors} For each $n \geq 3$, there exists  a constant, $M_n >0$, for which  the operator $\MN_{n, \, F}$  satisfies
		\begin{align*}
			\Lip \MN^x_{n, \, F}  & \leq \sup_{\theta \in \T^d_\sigma} |c(\theta)| +  M_n  \rho, \\
		\Lip \MN^y_{n, \, F} & \leq   k\, \sup_{\theta \in \T^d_\sigma}  |a_k(\theta)| +  M_n  \rho, \\
			\Lip \MN^\theta_{n, \, F} & \leq   p \, \sup_{\theta \in \T^d_\sigma}  |d_p(\theta)| +  M_n  \rho,
		\end{align*}
		where $\rho$ is the radius of the sector $S(\beta, \, \rho)$. 
	\end{lemma} 
	\begin{proof}
		We deal with the three components of $\MN_{n, F}$ separately. 
		First we prove the bound for  $\Lip \MN_{n, F}^x$. Let $f, \, \tilde f \in \mathcal{B}_{\alpha} \subset \MW_n^\times$. We have, 
		\begin{align*}
			\MN_{n, F}^x (f) - 	\MN_{n, F}^x (\tilde f) 
			& =  (\MK_n^y + f^y) \, \int_0^1 Dc \circ (\MK_n^\theta + \tilde f^\theta  + s(f^\theta - \tilde f^\theta)) \, ds    \, (f^\theta - \tilde f^\theta) \\
			& \quad + c \circ (\MK_n^\theta + \tilde f^\theta) \, (f^y - \tilde f^y).
		\end{align*}
		We can then bound, for some $M_n >0$,
		\begin{align*}
			\| (\MK_n^y + f^y) \, \int_0^1 & Dc \circ  (\MK_n^\theta + \tilde f^\theta  + s(f^\theta - \tilde f^\theta)) \, ds    \, (f^\theta - \tilde f^\theta) \|_{n+k-1} \\
			& \leq  \sup_{\T^d_\sigma} |Dc(\theta)| \, \sup_{S \times \T^d_{\sigma'}} \frac{1}{|u|^{n+k-1}} |\MK_n^y(u, \theta) + f^y(u, \theta)| |f^\theta(u, \theta) - \tilde f^\theta (u, \theta)| \\
			&   \leq \|f^\theta - \tilde f^\theta \|_{n+2p-k-1} \, \|\MK_n^y + f^y\|_{k+1} \,\sup_{\T^d_\sigma} |Dc(\theta)| \,  \rho^{2p-k+1} \\
			& \leq M_n \, \rho^{2p-k+1} \, \|f^\theta - \tilde f^\theta \|_{n+2p-k-1}.
		\end{align*}
		On the other hand,
		\begin{align*}
			\|c \circ (\MK_n^\theta + \tilde f^\theta) \, (f^y - \tilde f^y) \|_{n+k-1} & = \sup_{S \times \T^d_{\sigma '}} | c \circ (\MK_n^\theta + \tilde f^\theta) (u, \theta)| \frac{   |f^y (u, \theta) - \tilde f^y (u, \theta)|}{|u|^{n+k-1}} \\
			& \leq \sup_{\T^d_\sigma} |c (\theta)| \, \|f^y  - \tilde f^y\|_{n+k-1}, 
		\end{align*}
		and thus, we obtain 
		\begin{align*}
			\| \MN^x_{n,   F} (f) - \MN^x_{n,   F} (\tilde{f})\|_{n+k-1}  \leq  (\sup_{\T^d_\sigma} |c(\theta)| + M_n \rho ) \max \{ \|f^y  - \tilde f^y\|_{n+k-1}, \, \|f^\theta  - \tilde f^\theta\|_{n+2p-k-1}\},
		\end{align*}
		that is, 
		$
		\Lip \MN^x_{n, \, F}   \leq \sup_{\T^d_\sigma} |c(\theta)| + M_n \rho.$
		
		Next we consider $\MN_{n, F}^y$. We write, for $f, \, \tilde f \in \mathcal{B}_\alpha \subset \MW_n^\times$,
		\begin{align}
			\begin{split} \label{dif_N_tor}
				& \MN^y_{n,   \, F} (f) - \MN^y_{n,  \,  F} (\tilde f) = T_1^y + T_2^y,
			\end{split}
		\end{align}
		where 
		\begin{equation*} \label{part_dominant_Ny}
		T_1^y =
	a_k \circ (\MK_n^\theta  + f^\theta)(\MK_n^x + f^x)^k - a_k \circ (\MK_n^\theta  + \tilde f^\theta)(\MK_n^x + \tilde f^x)^k  \, \in \MW_{n+2k-2},
	\end{equation*}
		and $$
		T_2^y = \int_0^1 DA \circ \xi_s \, ds \, (f-\tilde f) \, \in \MW_{n+2k-1},
		$$
		where we have defined, for $s \in [0,1]$,
		$$
		\xi_s = \xi_s(f, \tilde{f}) = \MK_n + \tilde f + s(f - \tilde f) \in \MW_{2} \times \MW_{k+1} \times (\MW_{0})^d.
		$$
Note that indeed we have
		$
		\xi_s^x (u, \theta) = u^2 + O(|u|^3)$,  $ \xi_s^y (u, \theta) = \ol K_{k+1}^y u^{k+1} + O(|u|^{k+2})$, and $\xi_s^\theta (u, \theta) = \theta + O(|u|)$,
		since the presence of $f$ does not affect the lowest order terms of $\xi_s$, and since the coefficients depending on $\theta$ of $\MK_n(u, \theta)$ are bounded for $\theta \in \T^d_{\sigma'}$, as a consequence of the small divisors lemma.
		
	Since $T_1^y$ contains the leading terms of \eqref{dif_N_tor},	it is sufficient to bound the norm $\| T_1^y \|_{n+2k-2}$ to obtain the required estimate for \eqref{dif_N_tor}.  We write
		\begin{align*}
				T_1^y & =  a_k \circ (\MK_n^\theta  + f^\theta)(\MK_n^x + f^x)^k  - a_k \circ (\MK_n^\theta   + \tilde f^\theta)(\MK_n^x + \tilde f^x)^k \\
				&= a_k \circ (\MK_n^\theta + f^\theta) \, k \int_0^1   (\xi_s^x)^{k-1} \, ds (f^x - \tilde f^x)  + (\MK_n^x + \tilde f^x)^k  \int_0^1 D a_k \circ (\xi_s^\theta)  \, ds \,  (f^\theta - \tilde f^\theta) 
		\end{align*}
		and decomposing the last expression as $T_{11}^y + T_{12}^y$
in the obvious way,  		we have 
		\begin{align*}
			\| T_{11}^y\|_{n+  2k-2} &  \leq 	\| a_k \circ (\MK_n^\theta + f^\theta) \,  k  \int_0^1  (\xi_s^x)^{k-1} \, ds \,   \|_{2k-2} \,  \|(f^x - \tilde f^x)\|_{n} \\
			& \leq k \, \sup_{\T^d_\sigma}  \, |a_k(\theta)| \,   \sup_{s \in [0,   1]} \, \sup_{S \times \T^d_{\sigma' }} \, \frac{1}{|u|^{2k-2}} \, |\xi_s^x(u, \theta)|^{k-1}  \, \|f^x- \tilde f^x\|_n \\
			&    \leq  k \,  \sup_{\T^d_{\sigma}} \,   |a_k (\theta)|(1+ M_n \rho)\, \|f^x- \tilde f^x\|_n,
		\end{align*}
		and  $	\| T_{12}^y\|_{n+2k-2}
			 \leq M_n \, \rho^{2p-k+1} \, \|f^\theta - \tilde f^\theta \|_{n+2p-k-1}, $
		where $2p-k+1 \geq 1$.
		
		Putting together the obtained bounds, we have
		\begin{align*}
			\|\MN^y_{n,   \, F} (f) - \MN^y_{n,  \,  F} (\tilde f)\|_{n+2k-2} & \leq \|a_k(\MK_n^\theta  + f^\theta)(\MK_n^x +  f^x)^k - a_k(\MK_n^\theta  + \tilde f^\theta)(\MK_n^x + \tilde f^x)^k\|_{n+2k-2} \\
			& \quad +  \| T_2^y \|_{n+2k-2} \\ & \leq  (k \,  \sup_{\T^d_{\sigma}}  \,  |a_k (\theta)| + M_n \rho)\, \|f- \tilde f\|_{\MW_n^\times}.
		\end{align*}

		Finally we prove the result for  $\MN_{n, F}^\theta$ in an analogous way as for $\MN_{n, F}^y$.		
		Here we have, for each $f , \tilde f  \in \mathcal{B}_\alpha \subset \MW_n^\times$,
		\begin{align}
			\begin{split} \label{dif_N_tor_2}
				& \MN^\theta_{n,   \, F} (f) - \MN^\theta_{n,  \,  F} (\tilde f) = T_1^\theta + T_2^\theta, 
			\end{split}
		\end{align}
		where 
		$$
		T_1^\theta =	d_p \circ(\MK_n^\theta  + f^\theta)(\MK_n^x + f^x)^p - d_p \circ (\MK_n^\theta  + \tilde f^\theta)(\MK_n^x + \tilde f^x)^p  \, \in \MW_{n+2p-2},
			$$
		and 
		$$
		T_2^\theta = 	\int_0^1 DB \circ \xi_s  \, ds \, (f - \tilde f) \, \in \MW_{n+2p-1}.
		$$
		Since $T_1^\theta$ contains the leading terms of \eqref{dif_N_tor_2} we look for a bound for $\|T_1^\theta\|_{n+2p-2}$. 
		We have
		\begin{align*}
			T_1^\theta &= d_p \circ (\MK_n^\theta  + f^\theta)(\MK_n^x + f^x)^p  - d_p \circ (\MK_n^\theta  + \tilde f^\theta)(\MK_n^x + \tilde f^x)^p \\
			& = d_p \circ (\MK_n^\theta + f^\theta) \, p \int_0^1   (\xi_s^x)^{p-1} \, ds (f^x - \tilde f^x)  + (\MK_n^x + \tilde f^x)^p   \int_0^1 D d_p \circ (\xi_s^\theta)  \, ds (f^\theta - \tilde f^\theta).
		\end{align*}
	We decompose $T_1^\theta   = T_{11}^\theta + T_{12}^\theta$. We have 
			\begin{align*}
			\| T_{11}^\theta \|_{n+  2p-2} 
			& \leq 	\| d_p \circ (\MK_n^\theta + f^\theta) \,  p  \int_0^1  (\xi_s^x)^{p-1} \, ds \,   \|_{2p-2} \,  \|f^x - \tilde f^x\|_{n} \\
			&  \leq \sup_{s \in [0,   1]} \, \sup_{S \times \T^d_{\sigma' }} \, \frac{1}{|u|^{2p-2}} \,(p \, |d_p \circ (\MK_n^\theta + f^\theta)(u, \theta)|\,  |\xi_s^x(u, \theta)|^{p-1} ) \, \|f^x- \tilde f^x\|_n \\
			&    \leq  p \,  \sup_{\T^d_{\sigma}}  \,  |d_p (\theta)|(1 + M_n \rho)\, \|f^x- \tilde f^x\|_n, 
		\end{align*}
		and similarly, $	\|T_{12}^\theta \|_{n+2p-2} 
		\leq M_n \, \rho^{2p-k+1} \, \|f^\theta - \tilde f^\theta \|_{n+2p-k-1}$.
The term $T_2^\theta$ is of higher order. We  have $	\| T_2^\theta \|_{n+  2p-2}\le \rho\,  \| T_2^\theta \|_{n+  2p-1} $.  
With these estimates we get the bound for $\Lip  \MN^\theta_{n, \, F}$ claimed in the statement. 
	\end{proof}
	
	Next, we introduce some more operators.
	
	\begin{definition}
		For $n > 2p-k-1$, we denote by $\MS_{n, R}^\times : \MW_n^\times \to \MW_n^\times$ the linear operator defined component-wise as $ \MS_{n, \, R}^\times  = (\MS_{n,\, R}, \, \MS_{n+k-1,\, R}, \, ( \MS_{n+2p-k-1,\, R})^d)$.
	\end{definition}
	
	\begin{remark}
		Since the components of $\MS_{n, \, R}^\times$ are uncoupled, a right inverse 
		$(\MS_{n, R}^\times)^{-1} :\MW_{n+k-1}^\times \to \MW_n^\times  $
		 is given by 
        $$
		(\MS_{n, R}^\times)^{-1} = (\MS_{n, \,  R}^{-1}, \, \MS_{n+k-1, \, R}^{-1}, \, (\MS_{n+2p-k-1, \, R}^{-1})^d).
		$$
	\end{remark}
	
	\begin{definition} \label{def_T_analitic} 	Let $F$ be the holomorphic extension of an analytic map of the form \eqref{forma_normal_tor} satisfying the hypotheses of Theorem \ref{teorema_analitic_tors}.
		Given $n \geq 3$, we define $\MT_{n, \, F} : \mathcal{B}_\alpha \subset \MW_n^\times \to \MW_n^\times$ by 
		$$
		\MT_{n,  F} = (\MS_{n, R}^\times)^{-1} \circ \MN_{n,  F}.
		$$
	\end{definition}
	
	Using  the above operators, equations \eqref{eqdelta_analitic_tors} can be written as 
	\begin{equation*} 
		\MS_{n, \, R}^\times \, \Delta = \MN_{n, \, F} (\Delta ).
	\end{equation*}
	
	\begin{lemma} \label{lema_contraccio_tors}
		There exist $m_0 > 0$ and $\rho_0 >0$ such that if $\rho < \rho_0$ and  $n \geq m_0$, we have $\MT_{n, \, F} (\mathcal{B}_\alpha) \subseteq \mathcal{B}_\alpha$ and $\MT_{n , \, F}$ is a contraction operator in $\mathcal{B}_\alpha$. 
	\end{lemma}
	\begin{proof}
		By the definition of $\MT_{n,  F}$ and the norm in $\MW_n^\times$,
		\begin{align}
			\begin{split} \label{formula_lip_T}
				\Lip  \MT_{n,  \, F}  \leq \max \Big\{ & \|(\MS_{n, \, R})^{-1}\| \, \Lip \MN^x_{n, \,  F}, \, 
				\|(\MS_{n+k-1, \,  R})^{-1}\| \, \Lip \MN^y_{n,  F}, \\
				&  \|(\MS_{n+2p-k-1, \,  R})^{-1}\| \, \Lip \MN^\theta_{n, \,  F}\Big\}.
			\end{split}
		\end{align}
From \eqref{formula_lip_T} and the estimates obtained in Lemmas \ref{invers_S_tor}  and \ref{lema_N_tors}, given 
		$\mu \in (0,   (k-1) |\ol R_k^x | \cos \kappa)$, with $\kappa =  \frac{k-1}{2}\beta$,  there is $\rho_0 > 0$ such that for $\rho \in (0, \rho_0)$ we have the bound
		\begin{align*}
			\Lip \MT_{n, \,  F}   \leq  & \max \,  \Big\{  (\rho^{k+1} + \tfrac{1}{\mu} \tfrac{k-1}{n})(\sup_{T^d_{\sigma}} \,  |c(\theta)| + M_n \rho),    \\
			& \qquad (\rho^{k+1} + \tfrac{1}{\mu} \tfrac{k-1}{n+k-1})(\sup_{T^d_{\sigma}} \,  |a_k(\theta)| + M_n \rho), \, (\rho^{k+1} + \tfrac{1}{\mu} \tfrac{k-1}{n+2p-k-1})(\sup_{T^d_{\sigma}} \,  |d_p(\theta)| + M_n \rho ) \Big\} ,
		\end{align*}
		taking $S(\beta, \rho) \times \T^d_{\sigma' }$ as the domain of the functions of $\mathcal{B}_\alpha$.

		Then, choosing $\rho  < \rho_0$ small enough, it is clear that one can chose $m_0$ such that, for $n \geq m_0$, one has
		$\Lip  \MT_{n, \,  F}    < 1$. 
		
		Next we prove that one can find $\tilde \rho_0 > 0$, maybe smaller than $ \rho_0 $, such that taking $S(\beta, \rho) \times \T^d_{\sigma'}$, with $\rho  < \tilde \rho_0$ as the domain of the functions of $\mathcal{B}_\alpha$, then $\MT_{n, F}$ maps $\mathcal{B}_\alpha$ into itself.

		For each $f \in \mathcal{B}_\alpha$ we can write
		\begin{align*}
			\|\MT_{n, \, F}  (f) \|_{\MW_n^\times}  & \leq \| \MT_{n, \, F}  (f) -  \MT_{n, \, F}  (0)\|_{\MW_n^\times}  + \|\MT_{n, \, F}  (0)\|_{\MW_n^\times} \\
			& \leq  \alpha \, \Lip \MT_{n, \, F} + \|\MT_{n, \, F}  (0)\|_{\MW_n^\times}.
		\end{align*}
From the definitions of  $\MT_{n,  F} $ and $\MN_{n,  F}$  we have, for each $n \in \N$,
		$$
		\MT_{n,  F}(0) = (\MS_{n, \, R}^\times)^{-1} \circ \MN_{n, \, F} ( 0) = (\MS_{n, \, R}^\times)^{-1} \,  \MG_n.
		$$
		Also, we have $\MG_n = (\MG_n^x, \, \MG_n^y, \, \MG_n^\theta) \in \MW_{n+k} \times \MW_{n+2k-1} \times (\MW_{n+2p-1})^p$, and thus,  for every $\varepsilon > 0$, there is $\rho_n>0$ such that for $\rho< \rho_n$ one has
		\begin{align*}
			\|\MT_{n, \, F} (0)\|_{\MW_n^\times}  
			\leq \| (\MS_{n, \, R}^\times)^{-1} \| \, \max \{  \| \MG_n^x \|_{n+k-1}, \,  \| \MG_n^y \|_{n+2k-2}, \,  \| \MG_n^\theta \|_{n+2p-2}  \}   < \varepsilon.
		\end{align*}
		Moreover, since we have $\Lip \MT_{n,  F} <1$, we can take $\rho_n$ such that 
		$ \alpha \,  \Lip \MT_{n, \, F}  + \|\MT_{n, \, F} (0) \|_{\MW_n^\times} \leq \alpha 
		$,
		and then for every $\rho < \rho_n$ one has  $\MT_{n, \, F} (\mathcal{B}_\alpha) \subseteq \mathcal{B}_\alpha$. We have to take $\varepsilon \le \alpha (1-\Lip \MT_{n, \, F} )$.
	\end{proof}

\subsection{The case of vector fields} \label{sec-funcional-vf}

In this setting, we consider approximations  $ \MK_n : \R \times \T^d \times \R \times \Lambda \to \R^2 \times  
\T^d$ and $\mathcal{Y}_n: \R \times \T^d \times \R  \times \Lambda \to \R \times \T^d$   of 
solutions of equation \eqref{conj_tors_edo} obtained in Section \ref{sec_algoritme_tors_edo} up to a high enough order. Then, keeping $Y=\mathcal{Y}_n$ fixed, we look for a correction $\Delta: [0, \, \rho)  \times \T^d \times \R  \times \Lambda \to \R^2  \times \T^d$, for some $\rho>0$, of $\MK_n$, analytic on $(0,\rho) \times \T^d  \times \R \times \Lambda$, such that the pair $K= \MK_n + \Delta $, $Y=\mathcal{Y}_n$  satisfies the invariance condition
\begin{equation} \label{eq_delta_analitic_tors_edo}
	X \circ (\MK_n + \Delta, t) - \partial_{(u, \theta)} (\MK_n + \Delta) \cdot Y - \partial_t (\MK_n + \Delta) = 0.
\end{equation}
To be able to deal with equation \eqref{eq_delta_analitic_tors_edo} in a suitable space of analytic functions, we rewrite the vector field \eqref{fnormal_tors_edo1} in terms of its hull function $\check X(x, y, \theta, \tau , \lambda) = X (x, y, \theta, t, \lambda)$,  with  $\tau = \nu t$ and $\nu \in \R^{d'}$, and similarly for the functions that appear in its  components. Hence, the corresponding differential equation reads
\begin{equation} \label{fnormal_edo_skew}
	\begin{pmatrix}
		\dot x \\
		\dot y \\
		\dot \theta 
	\end{pmatrix} =
	\begin{pmatrix}
		\check c(\theta, \tau, \lambda) y \\
		\check a_k(\theta, \tau, \lambda)x^k + \check A(x, y, \theta, \tau, \lambda)  \\
		\omega + \check d_p(\theta, \tau, \lambda)x^p + \check B(x, y, \theta, \tau, \lambda) 
	\end{pmatrix},
\end{equation}
where  $\check c : \T^d \times \T^{d'} \times \Lambda \to \R$, $\check c(\theta , \tau, \lambda ) = c (\theta, t, \lambda)$, and similarly for the other quantities. Now the vector field $\check X$ is defined in a domain of the form $U \times \T^{d+d'}$, and thus the new variables $(\theta, \tau)$ can be thought as angles.

We also introduce 
$$
\check \MK_n (u, \theta, \tau, \lambda) = 
\MK_n(u, \theta, t, \lambda) , \qquad \check Y (u, \theta, \tau, \lambda) =  Y (u, \theta, t, \lambda),  
$$ and 
$$
J (u, \theta, \tau, \lambda) = 
\begin{pmatrix}
	\check Y (u, \theta, \tau, \lambda) \\
	\nu
\end{pmatrix}.
$$
Therefore, equation \eqref{eq_delta_analitic_tors_edo} can be written as 
\begin{equation} \label{analitic_hipotesi_compacte}
	\check X \circ (\check \MK_n + \Delta, \tau) - D (\check \MK_n + \Delta) \cdot J = 0,
\end{equation}
and then we look for a solution $\Delta = \Delta (u, \theta, \tau, \lambda)$ with  $\Delta: [0, \, \rho)  \times \T^d \times \T^{d'}  \times \Lambda \to \R^2  \times \T^d$.

The proofs of Theorems \ref{teorema_analitic_tors_edo} and \ref{teorema_posteriori_tors_edo} are organized in a similar way as the ones of Theorems \ref{teorema_analitic_tors} and \ref{teorema_posteriori_tors}. 
As for the case of maps, we will rewrite the equation for $\Delta $ as the fixed point equation.

From Proposition \ref{prop_tors_edo_sim}, given $n$ there exist a map $\MK_n$ and a vector field $Y= \mathcal{Y}_n$ such that 
\begin{equation*} 
	X \circ (\MK_n, t) - \partial_{(u, \theta)}\MK_n \cdot Y - \partial_t \MK_n = \MG_n, 
\end{equation*}
or equivalently,
\begin{equation} \label{analitic_hipotesi_tors_edo}
	\check X \circ (\check \MK_n, \tau ) - D \check \MK_n \cdot J  = \check \MG_n, 
\end{equation}
where $\check \MG_n(u, \theta, \tau, \lambda) = (O(u^{n+k}), \,O(u^{n+2k-1}), \, O(u^{n+2p-1}))$.
Since we are looking for a stable manifold we will take the approximations corresponding to $\check Y= \check \MY_n$ with the coefficient $\ol Y_k^x (\lambda)<0$.

Summarizing, 
we look for $\rho>0$ and a map $\Delta: [0, \, \rho)  \times \T^{d+d'}  \times \Lambda \to \R^2 \times \T^d$, analytic on $(0,\rho) \times \T^{d+d'} \times \Lambda$ satisfying \eqref{analitic_hipotesi_compacte},
where $\check \MK_n$ and $J$ satisfy \eqref{analitic_hipotesi_tors_edo}. Moreover, we ask $\Delta$ to  be of the form $\Delta = (\Delta^x, \Delta^y, \Delta^\theta) = (O(u^n), \, O(u^{n+k-1}), O(u^{n+2p-k-1}))$.

Similarly as in Section \ref{sec-funcional-maps}, we write 
$$
 \check P(x, y, \theta, \tau, \lambda) = \check a_k (\theta,
\tau , \lambda)x^k + \check A(x, y, \theta, \tau, \lambda ),
$$ 
$$
 \check Q(x, y, \theta, \tau, \lambda ) = \check d_p (\theta,\tau , \lambda)x^p + \check B(x, y, \theta,\tau, \lambda).
$$ 
Then, using \eqref{analitic_hipotesi_tors_edo} we  can rewrite \eqref{analitic_hipotesi_compacte} as 
\begin{align}  \begin{split}\label{eqdelta_analitic_tors_edo}
		D \Delta^x \cdot J &= \check \MK_n^y [\check c\circ (\check \MK_n^\theta + \Delta^\theta, \tau ) - \check c \circ (\check \MK_n^\theta, \tau) ]  + \Delta^y \, c\circ(\check \MK_n^\theta + \Delta ^\theta, \tau) +  \check \MG_n^x ,  \\
		D \Delta^y \cdot J   & = P \circ (\check \MK_n + \Delta, \tau) - P \circ (\check \MK_n , \tau)+ \check \MG_n^y ,\\
		D \Delta^\theta \cdot J  &  = Q \circ (\check \MK_n + \Delta, \tau) - Q \circ (\check \MK_n , \tau)+ \check \MG_n^\theta.
	\end{split}
\end{align}
To deal with equation \eqref{eqdelta_analitic_tors_edo} we introduce function spaces and operators adapted to the vector field setting.

\begin{definition}
	Given a sector $S = S(\beta, \, \rho)$, $\sigma > 0 $ and $n \in \N$, let $\MZ_n$,   be the Banach space 
	\begin{align*}
		\mathcal{Z}_{n}  = \bigg\{ f : S \times \T^{d+d'}_\sigma  \times \Lambda_\C & \rightarrow \C \, | \ f \text{ real analytic,}  \\  &\|f\|_n: = \sup_{(u, \theta, \tau, \lambda) \in S \times \T^{d+d'}_\sigma  \times \Lambda_\C} \frac{|f(u,\theta, \tau, \lambda)|}{|u|^n} < \infty \bigg\},
	\end{align*}
with the norm  $\|\cdot \|_n$. 
\end{definition}
Actually, it is exactly the same space as $\MW_n$ with the functions depending on 
$(\theta, \tau)\in \T^{d+d'}$ instead of depending on 
$\theta\in \T^{d}$.

As for the case of  maps, we endow the product spaces $\prod_i \MZ_i$ with the product norm and we define 
$$
	\MZ_n^\times = \MZ_n \times \MZ_{n+k-1} \times \MZ_{n+2p-k-1} ^d.
$$
Next, we set equation \eqref{eqdelta_analitic_tors_edo} in a space of holomorphic functions defined in a domain $S(\beta, \, \rho) \times \T^{d+d'}_{\sigma'} \times \Lambda_\C$, and we look for $\Delta$ being a real analytic function of  complex variables. Concretely, to solve equation \eqref{eqdelta_analitic_tors_edo}, we will consider $n $ big enough and we will look for a solution, $\Delta \in \mathcal{B}_\alpha \subset \MZ_n^\times$, for some $\alpha >0$.

 To determine suitable values for $\alpha$  we proceed in the same way as in the case of maps.
We take 
 $$
 \alpha = \min \,  \big\{ \tfrac{1}{2}, \, \tfrac{b}{2}, \,  \tfrac{\tilde \sigma}{2} \big\},
 $$ 
 where $b,\tilde \sigma, \sigma'$ and $\rho$ have the same meaning as there.

\begin{definition} \label{def_S_tor_edo}
	Let $k \geq 2$, $n \geq 0$, $\beta <  \tfrac{\pi}{k-1} $ and let 
	$
	J : S (\beta, \rho) \times \T^{d+d'}_{\sigma'}  \to \C \times\R^{d+d'}_{\sigma'}
	$
	be an analytic vector field of the form 
	\begin{equation} \label{R_def_S}
		J(u, \theta, \tau)  = ( Y_k u^k + O(u^{k+1}),
		\,  \omega  , \,  \nu),
	\end{equation}
	with $ Y_k < 0$ and such that the term $O(u^{k+1})$ does not depend on $(\theta, \tau)$.
	
	We define 
	$\MS_{n,J}: \MZ_n \rightarrow  \MZ_n$, as the linear operator given by
	\begin{align*}
		\MS_{n,J}\, f = Df \cdot J = \partial_u f \cdot J^x + \partial_\theta f \cdot \omega + \partial_\tau f \cdot \nu.
	\end{align*} 
\end{definition}
Note that this operator has a similar notation to a linear operator used in the map setting but it is different. 

The following lemma concerns the properties of the flows of vector fields of the form \eqref{R_def_S}.

\begin{lemma} \label{lema_prop_flux} Let $J (u, \theta, \tau) $ be as in Definition \ref{def_S_tor_edo} and let  $\varphi_s = (\varphi^u_s,\varphi^\theta_s, \varphi^\tau_s)$ be its flow. 
	Then, $\varphi_s$ has the form 
	\begin{equation*} \label{p_flux_1}
		\varphi_s(u, \theta, t) = (\varphi^u_s(u), \, \theta + \omega s, \, \tau + \nu s), 
	\end{equation*}
	and, for any fixed $\mu \in (0, \,\,  (k-1) |Y_k | \cos \kappa)$, with $\kappa =  \frac{k-1}{2}\beta$,  there exists $\rho_1\in (0,\rho]$ small enough such that
	 $\varphi_s^u(u) \in S(\beta, \rho_1)$ for all $u \in S(\beta, \rho_1)$ and $s \in [0, \infty)$. Moreover, 
 \begin{equation}  \label{p_flux_2}
		| \varphi^u_s(u) |  \leq \frac{|u|}{(1 + s \mu |u|^{k-1})^{\frac{1}{k-1}}}, \qquad   \forall \, u \in S(\beta, \, \rho_1), \quad \forall \, s \in [0, \infty).
	\end{equation}
\end{lemma}
\begin{proof}
By  definition, the time-$s$ flow of $J$ satisfies
\begin{equation} \label{def_flow_lema}
	\varphi_s(u, \theta, \tau) = (u, \theta, \tau) + \int_0^s J \circ \varphi_s \, ds,
\end{equation}
and thus,  we obtain 
$
\varphi_s^\theta(u, \theta, t)  = \theta + \omega s, \, \varphi_s^\tau(u, \theta, t)  = \tau + \nu s, 
$
and that $\varphi_s^u$ is independent of $\theta$ and $\tau$. 

Changing to complex polar coordinates, $u=re^{i\varphi}$, equation $\dot u =  Y_k u^k + O(u^{k+1})$ becomes
\begin{align}
\label{polars-r}
\dot r= Y_k \cos ((k-1)\varphi ) r^k + O(r^{k+1}), \\
\dot \varphi = Y_k \sin ((k-1)\varphi ) r^{k-1} + O(r^{k}).
\end{align}  	
In the domain $S(\beta, \rho) $,  $|(k-1)\varphi | <  \kappa <\pi/2$.
It is immediately checked that, if $\rho$ is small, on the boundary of $S$, the vector field points to the interior of $S$. Indeed, at $\varphi=\beta/2$, $\dot \varphi <0$; at $\varphi=-\beta/2$, $\dot \varphi >0$; and at $r=\rho$, $\dot r <0$.
For the last inequality we use that the  $O(r^{k+1})$ term in \eqref{polars-r} is less that $M r^{k+1} $ if $0<r<\rho $.  
We take $\tilde \mu$ such that $0<\mu<\tilde \mu < (k-1)|Y_k| \cos \kappa$ and  
$\rho_1 < \min\{ 1,\frac{\tilde \mu-\mu}{(k-1)M}\}$. 
Since $\cos ((k-1)\varphi ) > \cos \kappa >0$ we have 
$$
\dot r\le  Y_k \cos ((k-1)\varphi ) r^k + M r^{k+1}, \qquad  0<r<\rho_1.
$$
With the previous choices 
$$
\dot r\le  - \frac{\tilde \mu}{k-1}   r^k + M \rho r^k \le - \frac{ \mu}{k-1}   r^k  .
$$
Integrating the last inequality we obtain \eqref{p_flux_2}.
\end{proof}

The following lemma states  that   $\MS_{n,J}$ has a bounded right inverse and provides a bound of $\|\MS_{n,J}^{-1}\|$.

\begin{lemma} \label{invers_S_tor_edo}
	Given $k\geq 2$ and $n \geq 1$, the operator $\MS_{n,J} : \MZ_n \to \MZ_n $ has a bounded right inverse, 
	$$
	\MS_{n,J}^{-1} : \MZ_{n+k-1} \to \MZ_n,
	$$
	given by 
	\begin{equation} \label{sol_parabolic_tor_edo}
		\MS_{n,J}^{  -1} \, \eta =  - \int_{0}^\infty \, \eta \circ \varphi_s \, ds,  \qquad \eta \in \MZ_{n+k-1},
	\end{equation}
	where $\varphi_s$ denotes the time-$s$ flow of $J$.
	
	Moreover, for any fixed $\mu \in (0, \,\,  (k-1) |Y_k^x | \cos \kappa)$, with $\kappa =  \frac{k-1}{2}\beta$,
	there exists $\rho >0$  such that, taking $S(\beta, \, \rho) \times \T^{d+d'}_{\sigma'} $ as the domain of the functions  of $\MZ_{n+k-1}$,  
	we have 
	$$
	\|(\MS_{n,J})^{-1}\| \leq  \tfrac{1}{\mu} \, \tfrac{k-1}{n}.
	$$
\end{lemma} 
From Lemma \ref{lema_prop_flux} we have that  $\varphi_s^u(u, \theta, \tau) $ belongs to $S(\beta, \rho)$ for all $s \in [0, \infty)$. Then  clearly one has that $\varphi_s \in S(\beta, \rho) \times \T^{d+d'}_{\sigma'}$ and the composition  $\eta \circ \varphi_s$ is well defined for all $s\ge 0$.
Using again Lemma \ref{lema_prop_flux} we have, for $\rho$  small enough,
\begin{align*}
	|\eta \circ \varphi_s (u, \theta, \tau)| \, 
	& \leq\,   \|\eta \|_{n+k-1} \, \frac{1}{(\mu s)^{(1+\tfrac{n}{k-1})}}, \qquad \forall \,  (u, \theta, \tau) \in S \times \T^{d+d'}_{\sigma'}, \quad \forall \, s \in [0, \infty),
\end{align*}
so that the integral \eqref{sol_parabolic_tor_edo}  converges uniformly on $S \times \T^{d+d'}_{\sigma'} $.

\begin{proof} To show  that \eqref{sol_parabolic_tor_edo} is a formal expression for a right inverse of $\MS_{n,J}$,  we recall that $\varphi_s(u, \theta, \tau) = (\varphi^u_s (u), \theta + \omega s, \tau + \nu s)$ is the time-$s$ flow of $J$. By differentiating under the integral sign one has
	\begin{align*}
		\MS_{n,J} \circ (\MS_{n,J})^{-1}  \eta &  =  -\int_0^\infty   \partial_u (\eta \circ \varphi_s) \, ds \,   J^x - \int_0^\infty  \partial_\theta (\eta \circ \varphi_s)  \, ds \cdot \omega  - \int_0^\infty \partial_\tau (\eta \circ \varphi_s)  \, ds \cdot \nu  .
	\end{align*}
	Moreover, the following relations hold true, 
	\begin{align} \label{relacions-S-edo}
		\begin{split} 
			& \int_0^\infty \partial_\theta (\eta \circ \varphi_s)  \, ds \cdot \omega  = \int_0^\infty \partial_\theta \eta  \circ \varphi_s  \, \partial_s \varphi_s^\theta \, ds, \\
			&\int_0^\infty \partial_\tau (\eta \circ \varphi_s)   \, ds \cdot \nu = \int_0^\infty \partial_\tau \eta  \circ \varphi_s  \, \partial_s \varphi_s^\tau \, ds,
			\\
			& \int_0^\infty \partial_u (\eta \circ \varphi_s)  \, ds  \,  J^x = \int_0^\infty \partial_u \eta  \circ \varphi_s  \, \partial_s \varphi_s^u \, ds.
		\end{split}
	\end{align}
	Indeed, the first two equalities  above are  immediate. To prove the third one, observe that we have
	\begin{equation} \label{relacions-S-edo-2}
	\int_0^\infty \partial_u (\eta \circ \varphi_s) \, J^x \, ds     = \int_0^\infty \partial_u \eta \circ \varphi_s \, \partial_u \varphi_s^u \, J^x \, \frac{J^x\circ \varphi_s}{J^x\circ \varphi_s}\,  ds = 	\int_0^\infty g(s, u) \, h (s, u) \, ds,
	\end{equation}
	where $g(s,u) = \partial_u \eta \circ \varphi_s \, J^x \circ \varphi_s$ and $h(s,u) = \partial_u \varphi_s^u \frac{J^x}{J^x \circ \varphi_s}$. 

	We have that $\partial_s h(s,u) = 0$ and then $	h(s, u) =h (0, u)  = 1$ for all $s\ge 0$.

	Therefore, from \eqref{relacions-S-edo-2} we have
	$$
	\int_0^\infty \partial_u (\eta \circ \varphi_s) \, J^x \, ds   = 	\int_0^\infty g(s, u)  \, ds =  \int_0^\infty \partial_u \eta  \circ \varphi_s  \, \partial_s \varphi_s^u \, ds,
	$$
so that the third equality of \eqref{relacions-S-edo} is proved. Finally, using \eqref{relacions-S-edo} we obtain
	\begin{align*}
		\MS_{n,J} \circ (\MS_{n,J})^{-1}  \eta 
		&=  -\int_0^\infty \partial_s (\eta \circ \varphi_s ) \, ds = \eta \circ \varphi_0 -\lim_{s \to \infty} \eta \circ \varphi_s = \eta.
	\end{align*}
	
Now we check  that $\MS^{-1}_{n,J}$ is bounded on $\MZ_{n+k-1}$. 
	From  \eqref{sol_parabolic_tor_edo} and Lemma \ref{lema_prop_flux}, one has
	\begin{align*}
		\|(\MS_{n,J})^{-1} \,\eta\|_n & \leq \, \sup_{S\times \T^{d+d'}_{\sigma'} } \,  \frac{1}{|u|^n} \,  \int_{0}^{\infty} |(\eta \circ \varphi_s)(u, \theta, \tau)| \, ds \\
		& \leq \|\eta \|_{n+k-1} \, \sup_{ S} \,  \frac{1}{|u|^n} \, \int_{0}^{\infty} \, \bigg{(} \frac{|u|}{(1+s\mu|u|^{k-1})^{1/k-1}}\bigg{)}^{n+k-1} \, ds \,  \leq  \,  \frac{1}{\mu} \, \frac{k-1}{n} \,  \| \eta  \|_{n+k-1}.
	\end{align*}
\end{proof}

\begin{definition} \label{def_N_analitic_tor_edo} 
	Let $X$ be a vector field satisfying the hypotheses of Theorem \ref{teorema_analitic_tors_edo}, and let $\check X (x, y, \theta, \tau) = X(x, y, \theta, t)$, defined in $U_\C \times \T^{d+d'}_\sigma$. 
	Given $n \geq 3$,  we introduce $\MN_{n, X} = (\MN^x_{n,X},\, \MN^y_{n,X}, \, \MN_{n, X}^\theta): \mathcal{B}_\alpha \subset \MZ_n^\times \to \MZ_{n+k-1}^\times$,  by 
	\begin{align*}
		\MN_{n, X}^x (f) &= \check \MK_n^y [\check c\circ (\check \MK_n^\theta + f^\theta, \tau) - \check c \circ (\check \MK_n^\theta, \tau ) ]  + f^y \, \check c\circ(\check \MK_n^\theta + f^\theta, \tau) +  \check \MG_n^x ,  \\
		\MN_{n, X}^y(f) & = P \circ (\check \MK_n + f, \tau) - P \circ (\check \MK_n, \tau) + \check \MG_n^y ,\\
		\MN_{n, X}^\theta(f)  &  = Q \circ (\check \MK_n + f, \tau) - Q \circ (\check \MK_n, \tau) + \check \MG_n^\theta.
	\end{align*}
\end{definition}

With the previously introduced parameters,
the operators $\MN_{n, X}$ are Lipschitz.

\begin{lemma} \label{lema_N_tors_edo} For each $n \geq 3$, there exists  a constant, $M_n >0$, such that
	\begin{align*}
		\Lip \MN^x_{n,X}  & \leq \sup_{(\theta, \tau) \in \T^{d+d'}_\sigma} |\check c(\theta, \tau )| +  M_n  \rho, \\
		\Lip  \MN^y_{n,X} & \leq    k\, \sup_{(\theta, \tau) \in \T^{d+d'}_\sigma}  |\check a_k(\theta,\tau)| +  M_n  \rho, \\
		\Lip \MN^\theta_{n,X} & \leq   p \,\sup_{(\theta, \tau) \in \T^{d+d'}_\sigma}  | \check d_p(\theta,\tau )| +  M_n  \rho,
	\end{align*}
	where $\rho$ is the radius of the sector $S(\beta, \, \rho)$. 
\end{lemma} 
The proof is completely analogous to the one of Lemma \ref{lema_N_tors}, with the only difference that here the vector field $\check X$ and the functions of $\mathcal{B}_\alpha$ also depend on $\tau$. It will be omitted.

\begin{definition}
	For $n > 2p-k-1$, we denote by $ \MS_{n,J}^\times : \MZ_n^\times \to \MZ_n^\times $ the linear operator defined component-wise as $\MS_{n,J}^\times  = (\MS_{n,J}, \, \MS_{n+k-1,J}, \, ( \MS_{n+2p-k-1,J})^d )$.
\end{definition}

With these operators,  we can write equations \eqref{eqdelta_analitic_tors_edo} as 
\begin{equation*} 
\MS_{n,J}^\times \, \Delta = \MN_{n, X} (\Delta ).
\end{equation*}

Similarly as in Section \ref{sec-funcional-maps}, the inverse operator  $(\MS_{n,J}^\times)^{-1}$ is given by
 $$
(\MS_{n,J}^\times)^{-1} = (\MS_{n,J}^{-1}, \, \MS_{n+k-1,J}^{-1}, \, (\MS_{n+2p-k-1,J}^{-1})^d).
$$

\begin{definition} \label{def_T_analitic_edo} 
	Let $X$ be a vector field satisfying the hypotheses of Theorem \ref{teorema_analitic_tors_edo}, and let $\check X (x, y, \theta, \tau) = X(x, y, \theta, t)$, defined in $U_\C \times \T^d_\sigma \times \T^{d'}_\sigma$, $U_\C \subset \C^2$. 	Given $n \geq 3$, we define $\MT_{n, X} : \mathcal{B}_\alpha  \subset \MZ_n^\times \to \MZ_n^\times$ by 
	$$
	\MT_{n, X} =  (\MS_{n,J}^\times)^{-1} \circ \MN_{n, X}.
	$$
\end{definition}

\begin{lemma} \label{lema_contraccio_tors_edo}
	There exist $m_0 > 0$ and $\rho_0 >0$ such that if $\rho < \rho_0$, then, for every $n \geq m_0$, we have $\MT_{n, X} (\mathcal{B}_\alpha) \subseteq \mathcal{B}_\alpha$ and $\MT_{n, X}$ is a contraction operator in $\mathcal{B}_\alpha$. 
\end{lemma}
The proof is completely analogous to the one of Lemma \ref{lema_contraccio_tors} and will be omitted.

	\section{Proofs of the main results}
	 \label{sec-dems-tors}
	 
	 \subsection{The case of maps}

\begin{proof}[Proof of Theorem \ref{teorema_analitic_tors}]
		Let $m_0$ be the integer provided by Lemma \ref{lema_contraccio_tors}, and let $n_0 = \max \{ m_0, \, k + 1\}$. We take the maps  $\MK_{n_0}$ and $R = \MR_{n_0}$ given by Proposition \ref{prop_simple_tors}, which satisfy 
		\begin{equation*} \label{dem_analitic_polinomi}
			\MG_{n_0}(u, \theta) = F ( \MK_{n_0}(u, \theta )) - \MK_{n_0} (R(u, \theta))=(O(u^{n_0+k}), \, O(u^{n_0+2k-1}), \, O(u^{n_0+2p-1})).
		\end{equation*} 
We will look for $\rho >0$ and for a differentiable function $\Delta : [0, \rho) \times \T^d\to \R^2 \times \T^d$, $\Delta$ analytic in $(0, \rho) \times \T^d$, satisfying
\begin{equation} \label{eq_delta_dem_maps_tor}
	F \circ (\MK_{n_0} + \Delta) - (\MK_{n_0} + \Delta) \circ R = 0.
\end{equation}
Next, consider the holomorphic extension of $F$ to a neighborhood $U_\C \times \T^d_\sigma$ of $(0,0) \times \T^d$, where $U_\C \subset \C^2$ contains the closed ball of radius $b>0$  and take
		$\alpha= \min \, \{\tfrac{1}{2}, \, \tfrac{b}{2}, \, \tfrac{\tilde \sigma}{2},\}$, with $0 < \tilde \sigma < \sigma$.
		With this setting we rewrite \eqref{eq_delta_dem_maps_tor} as
		\begin{align*}  
			\Delta^x \circ R - \Delta^x  &= \MK_n^y [c\circ (\MK_n^\theta + \Delta^\theta) - c \circ \MK_n^\theta ]  + \Delta^y \, c\circ(\MK_n^\theta + \Delta ^\theta) +  \MG_n^x ,  \\
			\Delta^y \circ R -  \Delta^y  & = P \circ (\MK_n + \Delta) - P \circ \MK_n + \MG_n^y ,\\
			\Delta^\theta \circ R - \Delta^\theta &  = Q \circ (\MK_n + \Delta) - Q \circ \MK_n + \MG_n^\theta, 
		\end{align*}
		with $\Delta \in \mathcal{B}_\alpha \subset \MW_n^ \times$,	or using the  operators defined in the previous section,
		\begin{equation*} 
			\Delta = \MT_{n_0, \, F} (\Delta), \qquad \Delta \in \mathcal{B}_\alpha.
		\end{equation*}
By Lemma \ref{lema_contraccio_tors}, since $n_0 \geq m_0$, we have that $\MT_{n_0, \, 	F}$ maps $\mathcal{B}_\alpha$ into itself and is a contraction. Then it has a unique fixed point, $\Delta^\infty \in \mathcal{B}_\alpha$.
		Note that this solution is unique once $\MK_{n_0}$ is fixed. 
		Finally, $K = \MK_{n_0} + \Delta^\infty$ satisfies the conditions in the statement.
		
		The  $C^1$ character of $K$ at the origin follows from the order condition of $K$ at 0.
	\end{proof}

	\begin{proof}[Proof of Theorem \ref{teorema_posteriori_tors}]
		Let $m_0$ be the integer provided by Lemma \ref{lema_contraccio_tors}, and let $n_0 = \max \{ m_0, \, k + 1\}$. We distinguish two cases: 
 the value of $n$ given in the statement is such that $n<n_0$ or $n\ge n_0$. In the first first case we start looking for a better approximation $\MK^*$ of the form 
$$
\MK^* (u, \theta) = \wh K(u, \theta ) + \sum _{j=n+1} ^{n_0+1} \wh K_j(u, \theta),
$$
		with  
\begin{equation} \label{formulaK*}
		\wh K_j(u, \theta ) = \begin{pmatrix}
			\ol K_{j}^x u^j + \wt K_{j+k-1}^x (\theta) u^{j+k-1} \\
			\ol K_{j+k-1}^y u^{j+k-1} + \wt K_{j+2k-2}^y (\theta) u^{j+2k-2} \\
			\ol K_{j+2p-k-1}^\theta u^{j+2p-k-1} +	\wt K_{j+2p-2}^\theta (\theta) u^{j+2p - 2}
		\end{pmatrix}
\end{equation}
		and for 
		$$
		\MR^* (u, \theta) = \wh R(u, \theta ) + \sum _{j=n+1} ^{n_0+1} \wh R_j(u),
		$$
		with 
		\begin{align}\label{formulaR*}
			\wh R_j^x (u) = \begin{cases}
				\delta_{j, k+1} \ol R_{2k-1}^x t^{2k-1} & \quad \text{ if } n \leq k, \\
				0  & \quad \text{ if } n > k, 
			\end{cases} \qquad  \ \ \
			\wh R_j^\theta (u) = 0.
		\end{align}
In the secon case, when $n\ge n_0$, we take    		
$\MK^* = \wh K+ \wh K_{n+1} $ and $\MR^* = \wh R+ \wh R_{n+1} $, with 
$\wh K_{n+1}$ and $ \wh R_{n+1} $ again as in \eqref{formulaK*} and \eqref{formulaR*}, respectively. 		
We introduce
\begin{align*}
	n^* = \begin{cases}
		n_0 & \quad \text{ if } n <n_0, \\
		n+1  & \quad \text{ if } n \ge n_0.
	\end{cases} 
\end{align*}
		The coefficients  $\wh K_{j}   $ and $\wh R_{j} $, $ n+1\le j \le n_0$,  are obtained imposing the condition
		\begin{equation*} \label{approx-hyp-n0}
			F ( \MK^* (u, \theta) )-   \MK^* ( \mathcal{R}^*(u, \theta ) )=  (O(u^{n^*+k}), O(u^{n^*+2k-1}), \, O(u^{n^* + 2p-1})) .
		\end{equation*}
Indeed, proceeding as in Proposition \ref{prop_simple_tors},  we obtain these coefficients iteratively.  We denote   $ \MK_j(u, \theta ) = \wh K(u, \theta) + \sum_{m=n+1} ^j \wh K_m (u, \theta) $ 
		and $\mathcal{R}_j(u, \theta ) = \wh R (u, \theta )  + \sum_{m=n+1}^j \wh R_m(u) $ for $j\ge n+1$.  In the iterative step we have 
		\begin{equation*} \label{approx-hyp-j}
			F ( \MK_j (u , \theta )) -   \MK_j (  \mathcal{R}_j( u, \theta )) =  (O(u^{j+k}), O(u^{j+2k-1}), \, O(u^{j+2p-1})) .
		\end{equation*}
		Then, 
		\begin{align*}
			F(\MK_j(u, \theta ) + \wh K_{j+1} (u, \theta ))  - & (\MK_j+ \wh K _{j+1})\circ (\MR_j(u, \theta) + \wh R_{j+1}(u) )  \\
			= &  F(\MK_j (u, \theta )) - \MK_j (\MR_j(u, \theta) ) \\
			& + DF(\MK_j (u, \theta ))\wh K_{j+1}(u, \theta) - \wh K_{j+1} (\MR_j(u, \theta )+\wh R_{j+1} (u)) \\
			& + \int_0^1 (1-s)  D^2F(\MK_j(u, \theta ) + s \wh K_{j+1}(u, \theta) ) \, ds \, (\wh K_{j+1}(u, \theta))^{\otimes 2} \\
			&  -D \MK_j(\MR_j(u, \theta ) )\wh R_{j+1} (u) \\
			& - \int_0^1 (1-s)  D^2\MK_j (\MR_j(u, \theta )  + s \wh R_{j+1}(u)) \, ds \, (\wh R_{j+1}(u))^{\otimes 2}  . 
			\end{align*}
		The condition 
		\begin{equation*} \label{approx-hyp-tor}
			F ( \MK_{j+1} (u, \theta )) -  \MK_{j+1} ( \mathcal{R}_{j+1}(u, \theta )) =  (O(u^{j+k+1}), O(u^{j+2k}), \, O(u^{j+2p})) 
		\end{equation*}
		leads to  equations \eqref{SDequations_induccio_sim} and  \eqref{sist_lineal_tors_sim} in Proposition \ref{prop_simple_tors},
		which we solve in the same way. 
		
		From this point we can proceed as in the proof of Theorem \ref{teorema_analitic_tors} and look for $\Delta \in  \mathcal{B}_\alpha  \subset \MW_n^\times$ such that 
		the pair $K= \MK^* + \Delta $, $R=\mathcal{R}^*$
		satisfies
		$ F \circ K = K \circ R$. 
		
		Finally, for the map $K$, we also have
		\begin{align*}
			K(u, \theta ) - \wh K(u, \theta)  & = \MK^*(u, \theta) - \wh K(u, \theta) + \Delta(u, \theta) \\
			& =  \sum _{j=n+1} ^{n^*} \wh K_j(u, \theta) + \Delta(u, \theta) \\
			& = (O(u^{n+1}), O(u^{n+k}), \, O(u^{n+2p-k}))  +(O(u^{n^*}), O(u^{n^*+k-1}), \, O(u^{n^*+2p-k-1})).
		\end{align*} 
	Since $n^* \ge n+1$ we have  $n+2p - k \leq n^* + 2p - k - 1$, 
		and therefore, 
		$$
		K(u, \theta ) - \wh K(u, \theta) = (O(u^{n+1}), O(u^{n+k}), \, O(u^{n+2p-k})).
		$$
For the map $R$ we have
		\begin{align*}
			R(t, \theta) - \wh R(u, \theta) &  = 	\MR^*(u, \theta) - \wh R(u, \theta) = \sum _{j=n+1} ^{n^*} \wh R_j(u)  =\left\{ \begin{array}{ll}
				(  O(u^{2k-1}) ,  0 ) & \quad \text{  if }\; n\le k, \\
				(0, 0)  & \quad \text{ if } \; n>k.
			\end{array} \right.
		\end{align*}
	\end{proof}
	
\subsection{The case of vector fields}

\begin{proof}[Proof of Theorem \ref{teorema_analitic_tors_edo}]
	Let $m_0$ be the integer provided by Lemma \ref{lema_contraccio_tors_edo}, and let $n_0 = \max \{ m_0, \, k + 1\}$. We take the approximations $\MK_{n_0}$ and $Y = \MY_{n_0}$ given by Proposition \ref{prop_tors_edo_sim}, which satisfy 
	\begin{align*}
		\MG_{n_0}(u, \theta, t) & = X ( \MK_{n_0}(u, \theta, t ), t) -	\partial_{(u, \theta)}\MK_{n_0} (u, \theta, t) \cdot Y(u, \theta, t) - \partial_t \MK_{n_0} (u, \theta, t)\\
		& =(O(u^{n_0+k}), \, O(u^{n_0+2k-1}), \, O(u^{n_0+2p-1})).
	\end{align*} 
	We will look for $\rho >0$ and a function $\Delta : [0, \rho) \times \T^d \times \R \to \R^2 \times \T^d$, $\Delta$ analytic in $(0, \rho) \times \T^d \times \R$, satisfying
	\begin{equation} \label{eqdelta_dem_tor_edo}
		X \circ (\MK_{n_0} + \Delta, t) - \partial_{(u, \theta)} (\MK_{n_0} + \Delta) \cdot Y - \partial_t (\MK_{n_0} + \Delta) = 0.
	\end{equation}
Let $\check X (x, y, \theta, \tau) = X(x, y, \theta, t)$ be the hull function of $X$ and consider  the holomorphic extension of $\check X$ to a neighborhood $U_\C \times \T^{d+d'}_\sigma$ of   $(0,0)  \times \T^{d+d'}$, where $U_\C \subset \C^2$ contains the  closed ball of radius $b>0$, and we also take  
	$\alpha= \min \, \{\tfrac{1}{2}, \, \tfrac{b}{2}, \, \tfrac{\tilde \sigma}{2}\}$ with $0 < \tilde \sigma < \sigma$.
	This setting allows to rewrite \eqref{eqdelta_dem_tor_edo}  as
	\begin{align*}  
		D \Delta^x \cdot J &= \check \MK_n^y [\check c\circ (\check \MK_n^\theta + \Delta^\theta, \tau ) - \check c \circ (\check \MK_n^\theta, \tau) ]  + \Delta^y \, c\circ(\check \MK_n^\theta + \Delta ^\theta, \tau) +  \check \MG_n^x ,  \\
		D \Delta^y \cdot J   & = P \circ (\check \MK_n + \Delta, \tau) - P \circ (\check \MK_n , \tau)+ \check \MG_n^y ,\\
		D \Delta^\theta \cdot J  &  = Q \circ (\check \MK_n + \Delta, \tau) - Q \circ (\check \MK_n , \tau)+ \check \MG_n^\theta,
	\end{align*}
	with	$\Delta \in \mathcal{B}_\alpha \subset \MZ_n^\times$, or using the operators defined in the vector field setting,
	\begin{equation*}
		\Delta = \MT_{n_0,X} (\Delta), \qquad \Delta \in \mathcal{B}_\alpha.
	\end{equation*}
	By Lemma \ref{lema_contraccio_tors_edo}, since $n_0 \geq m_0$, we have that $\MT_{n_0,X}$ maps $\mathcal{B}_\alpha$ into itself and is a contraction. Then it has a unique fixed point, $\Delta^\infty \in \mathcal{B}_\alpha$.
	Note that this solution is unique once $\check \MK_{n_0}$ is fixed. 
	Finally we take $\wt \Delta^\infty (u, \theta, t) = \Delta^{\infty} (u, \theta, \tau)$, and then  $K = \MK_{n_0} + \wt \Delta^\infty$ satisfies the conditions in the statement.
	
Again, the  $C^1$ character of $K$ at the origin follows from the order condition of $K$ at $u= 0$.
\end{proof}

\begin{proof}[Proof of Theorem \ref{teorema_posteriori_tors_edo}] 
	The proof is completely analogous to the one of Theorem \ref{teorema_posteriori_tors}, taking into  account that now we are in the vector field setting. In the last step we use the same  argument as in the proof of Theorem \ref{teorema_analitic_tors_edo}.
\end{proof}

\section{Appendix}

\subsection{Proof of Theorem \ref{thm:unicitat}}

The proof consists in doing a number of changes of variables that transforms the map into a new map such that their $x$ and $y$-components are independent on the angles up to order $k$ included, and has the form of the maps studied in \cite{fontich99}.  
Then we can use the inequalities obtained in that paper to get the uniqueness. 

We write the map in the form 
\begin{equation} \label{forma_normal_tor-unicitat} 
	F
	\begin{pmatrix}
		x\\ y\\ \theta
	\end{pmatrix}
	= 
	\begin{pmatrix}
		x + c(\theta) y \\
		y + a_k(\theta)x^k + A_{k-1}(x, y, \theta)y + A_{k+1}(x, y, \theta)  \\
		\theta + \omega +d_p(\theta)x^p + B_{p-1}(x, y, \theta)y + B_{p+1}(x, y, \theta)
	\end{pmatrix},
\end{equation}	
where $A_{k-1}$ and $B_{p-1}$ are homogeneous polynomials of degree $k-1$ and $p-1$   in $x,\,y$, respectively, depending on $\theta$, and 
$A_{k+1}$ and $ B_{p+1}$ are function of order $k+1$ and $p+1$ respectively.
We recall that $2p>k-1$. Moreover, it is convenient to assume that $p\le k$. If not, terms of order 
$p$ can be put in a remainder of order $k+1$.

In this proof we will write $O_j$ for a term of order $j$ in the variables $x\,y $ and $O(x^\ell  y^m)$ for a term of order 
$x^\ell y^m$. Both terms may depend on $\theta$.

We start doing some steps of averaging. 
We consider changes of the form
\begin{align}
	\Phi_1(x,y,\theta) & = (x+\phi(\theta) x^\ell y^m, y, \theta), \label{canviphi1}\\ 
	\Phi_2(x,y,\theta) & = (x, y+\phi(\theta) x^\ell y^m, \theta), \label{canviphi2}\\
	\Phi_3(x,y,\theta) & = (x, y, \theta+\phi(\theta) x^\ell y^m)	\label{canviphi3},
\end{align}
to average the monomials of order $x^\ell y^m$ in the $x$, $y$ and $\theta $ components, respectively.  
These changes may introduce new terms of order bigger or equal than $\ell+m$ in any component. 
Below we do a study of the terms that appear.
First we do a change of the form \eqref{canviphi1} with $\ell=0$ and $m=1$
to average the term $c(\theta)y $. We obtain
\begin{align*}
	F^{(1)} 
	\begin{pmatrix}
		x  \\
		y \\
		\theta 
	\end{pmatrix}
	& =
	\begin{pmatrix}
		x + \phi(\theta) y + c(\theta) y - \phi(\hat \theta )(y+a_k(\theta)x^k+ yO_{k-1}+O_{k+1} )\\
		y + a_k(\theta)(x+\phi(\theta)  y)^k + A_{k-1}(\hat x, y, \theta)y + A_{k+1}(\hat x, y, \theta)  \\
		\theta + \omega +d_p(\theta)(x+\phi(\theta) y)^p + B_{p-1}(\hat x, y, \theta)y + B_{p+1}(\hat x, y, \theta)
	\end{pmatrix},
\end{align*}
where $\hat x = x + \phi(\theta)  y$ and 
$\hat \theta=\theta + \omega +d_p(\theta)(x+\phi(\theta) y)^p + B_{p-1}(\hat x, y, \theta)y + B_{p+1}(\hat x, y, \theta)$. 

We can rewrite the first component as 
\begin{align*}
	x & + \Big(\phi(\theta)  + c(\theta) -  \phi(\theta + \omega) \Big)y 
	+ \Big(\phi(\theta + \omega) - \phi(\hat \theta ) \Big)y
	- \phi(\hat \theta ) O(x^k) + yO_{k-1}+O_{k+1}, 
\end{align*}
and writing $c=\bar c + \tilde c$, by the small divisors lemma there exists 
a unique zero average function $\phi$ such that 
$$
\phi(\theta + \omega) - \phi(\theta)  = \tilde c(\theta)   
$$
and taking $\phi $ as such solution the first component becomes
\begin{align*}
	x & + \bar  c y 
	+ y (O_{p}+ O_{k-1})+O_{k} =: C_p( x, y, \theta)  y + C_k( x, y, \theta).
\end{align*}
The second and third components have the same structure and the same lower order terms 
$y + a_k(\theta)x^k $ and $ 
\theta + \omega +d_p(\theta)x^p$, respectively,  
as $F$. Thus we have
\begin{equation} \label{primer-pas-averaging} 
	F^{(1)}	\begin{pmatrix}
		x\\ y\\ \theta
	\end{pmatrix}
	= 
	\begin{pmatrix}
		x + \bar c y  +  C_p( x, y, \theta)  y + C_k( x, y, \theta) \\
		y + a_k(\theta)x^k + A^{}_{k-1}(x, y, \theta)y + A_{k+1}(x, y, \theta)  \\
		\theta + \omega +d_p(\theta)x^p + B_{p-1}(x, y, \theta)y + B_{p+1}(x, y, \theta)
	\end{pmatrix},
\end{equation}	
with new functions $A's$  and $B's$ of the same form as the ones in  
\eqref{forma_normal_tor-unicitat}.

The changes  $\Phi_1$,  $\Phi_3$ and  $\Phi_3$ applied to $F^{(1)}$ give new maps with the same structure and with coefficients averaged, provided we choose a suitable $\phi$. Concretely,
the change 	
$\Phi_1$, used for $\ell+m \ge p+1$, produces
\begin{align*}
	&\Phi_1^{-1} \circ F^{(1)} \circ \Phi_1
	\begin{pmatrix}
		x  \\
		y \\
		\theta 
	\end{pmatrix}
	\\
	= 
	& 
	\begin{pmatrix}
		x + \phi(\theta) x^\ell y^m + \bar c y ++ C_{p}(\hat x, y, \theta)y+ C_k(\hat x, y, \theta)- \phi(\hat \theta)(x+ \bar c y+O_{\ell+m} )^\ell(y+O(x^k)) ^m+O_{k+1} \\
		y + a_k(\theta)(x+\phi(\theta) x^\ell y^m)^k + A_{k-1}(\hat x, y, \theta)y + A_{k+1}(\hat x, y, \theta)  \\
		\theta + \omega +d_p(\theta)(x+\phi(\theta) x^\ell y^m)^p + B_{p-1}(\hat x, y, \theta)y + B_{p+1}(\hat x, y, \theta)
	\end{pmatrix},
\end{align*}
where $\hat x= x + \phi(\theta) x^\ell y^m$ and 
$\hat \theta = \theta + \omega +d_p(\theta)(x+\phi(\theta) x^\ell y^m)^p + B_{p-1}(\hat x, y, \theta)y + B_{p+1}(\hat x, y, \theta)$.

The change 	
$\Phi_2$, used for, used for $\ell+m = k$, produces
\begin{align*}
	&\Phi_2^{-1} \circ F^{(1)} \circ \Phi_2
	\begin{pmatrix}
		x  \\
		y \\
		\theta 
	\end{pmatrix}
	\\
	& =
	\begin{pmatrix}
		x + \bar c \hat y +C_{p}( x,\hat y, \theta)\hat y+ C_k( x, \hat y, \theta) \\
		y + \phi(\theta) x^\ell y^m+ a_k(\theta)x^k + A_{k-1}(x,\hat y, \theta)y + A_{k+1}( x, \hat y, \theta)  -\phi(\hat \theta) (x + \bar c \hat y )^\ell y ^m + O_{k+1}\\
		\theta + \omega +d_p(\theta)x^p + B_{p-1}( x, \hat y, \theta)\hat y + B_{p+1}(x, \hat y, \theta)
	\end{pmatrix},
\end{align*}
where $\hat y = y + \phi(\theta) x^\ell y^m$ and 
$\hat \theta= \theta + \omega +d_p(\theta)x^p + B_{p-1}( x, \hat y, \theta)\hat y + B_{p+1}(x, \hat y, \theta)$.

And the change 	
$\Phi_3$, used when $\ell+m \ge p$, produces
\begin{align*}
	&\Phi_3^{-1} \circ F^{(1)} \circ \Phi_3
	\begin{pmatrix}
		x  \\
		y \\
		\theta 
	\end{pmatrix}
	\\
	& =
	\begin{pmatrix}
		x + \bar c y + C_{p}(x, y, \hat \theta) y + C_{k}(x, y, \hat \theta) \\  
		y + a_k(\hat \theta)x ^k + A_{k-1}(x, y,\hat  \theta)y + A_{k+1}(x, y, \hat \theta)  \\
		\theta  + \phi(\theta) x^\ell y^m+ \omega +d_p(\hat \theta)x^p + B_{p-1}(x, y,\hat  \theta)y + B_{p+1}(x, y, \hat \theta)
		- \phi(\hat \theta+ \omega ) (x+\bar c y)^\ell y^m+O_{k} 
	\end{pmatrix},
\end{align*}
where $\hat \theta= \theta + \phi(\theta) x^\ell y^m$.

Now we do several changes of the form $\Phi_3$ to average the terms of order $p$ of the third component. 
This may introduce new terms of the same order but with a higher value of the exponent of $y$.  
In this procedure we use the small divisors lemma in a completely analogous way as we did to arrive to \eqref{primer-pas-averaging}. 
Thus, 
we start averaging the term $x^p$, then the term $x^{p-1}y $ and so on until the term $y^p$ which can be averaged without introducing new terms of order $p$. 
These changes introduce terms of order $yO_{2p}$ in the first component and terms of order $yO_{k+p-1}$ in the second one. Since $2p > k-1$, both kind of terms are of order $k+1$ or bigger.

Then we proceed doing changes of the form $\Phi_3$ and $\Phi_1$ to average the terms of order $p+1$ and higher starting with the terms $x^j$ and ending with $y^j$ when dealing with a degree $j$, $p+1\le j\le k$ in the first component. 
The changes $\Phi_1$, when considering the averaging of a term of order 
$x^\ell y ^m$ introduce new terms of order $x^{\ell+i} y ^{m-i}$, $i\ge 1$ (possibly depending on $\theta$). 
Moreover, when  $\ell+m= p+1$ introduce terms $yO_{k-1} $ in the third component. At order $\ell+m> p+1$ they introduce terms of order $O_{k+1}$ which can be forgotten. However they maintain the structure of the second component. Similarly, the changes $\Phi_3$, when averaging a term of order 
$x^\ell y ^m $ in the third component  add terms of order $x^{\ell +i} y^{m-i}$, $i\ge 1$ in the third component. Also, when $\ell+m= p $  they may introduce a term of order $k$ in the third component. That is why we proceed in the order indicated in the previous paragraph.     
Finally we do  changes of the form $\Phi_2$
to average the terms of ordre $k$ of the second component while do not change the already obtained terms of order less or equal than $k$.

After having done these changes we arrive to a map of the form
\begin{equation} \label{mapa-promitjat} 
	\wh F
	\begin{pmatrix}
		x\\ y\\ \theta
	\end{pmatrix}
	= 
	\begin{pmatrix}
		x + \bar c y +  \wh C_{p}(x,y)y+ c_k x^k+\wh C_{k+1}(x, y, \theta)  \\
		y + \bar a_k x^k + \wh A_{k-1}(x, y)y + \wh A_{k+1}(x, y, \theta)  \\
		\theta + \omega +\bar d_px^p + \wh B_{p-1}(x, y)y + \wh B_{k}(x, y, \theta)
	\end{pmatrix}.
\end{equation}

Now we do a change $(x, y, \theta)\mapsto (\bar c  x, y, \theta)$ which maintains the same form and changes the constant $\bar c$ to $ 1$.

Next, we consider the related two-dimensional map (independent of the angles)
\begin{equation} 
	G
	\begin{pmatrix}
		x\\ y
	\end{pmatrix}
	= 
	\begin{pmatrix}
		x + y +   \wh C_p(x,y)y + c_k x^k
		\\
		y + \bar a_k x^k + \wh A_{k-1}(x, y) y
	\end{pmatrix},
\end{equation}	
and we do the changes to transform it to the normal form 
$$
N
\begin{pmatrix}
	x\\ y
\end{pmatrix}
= 
\begin{pmatrix}
	x + y 
	\\
	y +   a_k x^k +\dots + b_\ell x^{\ell-1} y + \dots 
\end{pmatrix}
+ O_{k+1}
$$
given in \cite{fontich99}. The change to this normal form is known, and it is  described in detail in  \cite{fontich99}.
To arrive to the normal form one has to do
a sequence of changes of the form 
\begin{equation}\label{canviC}
	C
	\begin{pmatrix}
		\xi \\ \eta 
	\end{pmatrix}
	= 
	\begin{pmatrix}
		\xi + \Phi(\xi,  \eta)
		\\
		\eta  + \Psi(\xi,  \eta)
	\end{pmatrix},
\end{equation}
where $\Phi$
and $\Psi $ are homogeneous polynomials of degree $j\ge 2$ to remove 
as many monomials of degree $j$ as possible. 
The inverse is  
\begin{equation}\label{canviCinversa}
	C^{-1}
	\begin{pmatrix}
		\xi \\ \eta 
	\end{pmatrix}
	= 
	\begin{pmatrix}
		\xi - \Phi(\xi,  \eta)
		\\
		\eta  - \Psi(\xi,  \eta)
	\end{pmatrix} + O_{2j-1}.
\end{equation}
We will use changes of order $j\ge p+1$. Then $2j-1\ge 2p+1 >k$, so that in our computations these terms do not play any role.

With these changes one can remove all terms of order $j$ except, in general, the terms $x^j$ and $x^{j-1} y $ in the second component of the map.  
We claim that in our case we can remove all terms of the first component (except de linear ones) without adding new terms in the second component. Indeed, 
assume inductively that $G$ has the form
$$
G
\begin{pmatrix}
	x\\ y
\end{pmatrix}
= 
\begin{pmatrix}
	x + y + \sum_{i=j}^k E_{i}(x,y) 
	\\
	y + \bar a_k x^k +  A_{k-1}(x, y)  y
\end{pmatrix} + O_{k+1},
$$
where $E_{i}$ is a homogeneous polynomial of degree $p+1\le i< k$.
It very important to note that when $i\le k-1$, $E_{i}(x,y) = y O_{i-1}$.
Doing a change of the form \eqref{canviC} we obtain 
\begin{align*}
	C^{-1} & \circ G \circ C
	\begin{pmatrix}
		x\\ y
	\end{pmatrix}
	= 
	\begin{pmatrix}
		x + \Phi + y +\Psi - \Phi (x  + y  , y )+ \sum_{i=j}^k E_{i}(x   ,y )
		\\
		y  +\Psi -  \Psi (x   + y  , y )
	\end{pmatrix}	+  O_{k+1}	.
\end{align*}
If we want to remove all terms of order $j$ we need to solve
\begin{align*}
	\Phi(x   ,y )  +\Psi(x   ,y ) - \Phi (x  + y  , y )+ E_j(x   ,y )
	& =0,\\
	\Psi(x   ,y ) -  \Psi (x   + y  , y )
	& =0.
\end{align*}
Since these equations are for homogeneous polynomiasl of degree $j$, actually we have a linear system for the coefficients of $\Phi$ and $\Psi$, 
given the coefficients of $E_j$. This system is studied in detail in \cite{fontich99}. We emphasize that here $E_j$ does not contain the term $x^j$. Then we can take $\Psi =0$ and then solve the first equation taking into account the mentioned property. 

When $j=k$ we have 
\begin{align*}
	\Phi(x   ,y )  +\Psi(x   ,y ) - \Phi (x  + y  , y )+ E_k(x   ,y )
	& =0,\\
	\Psi(x   ,y ) -  \Psi (x   + y  , y ) + a_kx^k + \wh A_{k-1}(x,y)y
	& =0.
\end{align*}
Now we are in the general situation and we get that the order $k$ terms of the normal form are
$$
\begin{pmatrix}
	0\\
	a_kx^k +b_{k-1}x^{k-1}y
\end{pmatrix}.
$$
In Section 4 of \cite{fontich99} it is proved that the (parabolic) stable manifold of the map 
$$
\begin{pmatrix}
	x+y \\
	y+	a_kx^k +b_{k-1}x^{k-1}y 
\end{pmatrix} + O_{k+1}
$$
is unique. Let $C$ be the change that  puts $G$ in the normal form $N$.
Consider the change $(x,y,\theta ) \mapsto (C(x,y),\theta) $ that tranforms 
the map $\wh F$ in \eqref{mapa-promitjat} into a new map 
\begin{equation} \label{mapa-promitjatinormalitzat} 
	\check  F
	\begin{pmatrix}
		x\\ y\\ \theta
	\end{pmatrix}
	= 
	\begin{pmatrix}
		x + \bar c y +  \check  C_{k+1}(x, y, \theta)  \\
		y + \bar a_k x^k + b_{k-1}x^{k-1} y + \check  A_{k+1}(x, y, \theta)  \\
		\theta + \omega +\bar d_px^p + \check  B_{p-1}(x, y)y + \check  B_{k}(x, y, \theta)
	\end{pmatrix},
\end{equation}
The remainders of the first and second components are of order $k+1$ and uniformly bounded with respect to $\theta$. 
Then, all bounds of Section 4 of  \cite{fontich99} are also valid here
for the first and second components of the iterates of $\check F$ and we have uniqueness of the stable manifold of $\check F$. 
Undoing the (close to the identity) change, the stable manifold of $F$ is also unique.

 \subsection{Proof of Theorem \ref{teorema_analitic_helicoure}}
 \label{App-A1}

The proof of Theorem \ref{teorema_analitic_helicoure} is completely analogous to the one of Theorem \ref{teorema_analitic_tors_edo}. However, for the convenience of the reader we will sketch here an overview of the proof and the spaces and operators that have to be used.

An analogous argument to the one of the proof of Proposition \ref{prop_tors_edo_sim} provides the expressions of the first coefficients of the parameteriaztions of $K$ and $Y$, namely those given in \eqref{coefs1_th_helicoure}.

Proceeding inductively  as described in the proof of Theorem \ref{prop_tors_edo_sim} we obtain that  given $n$ there exist  $\MK_n$ and $ \mathcal{Y}_n = Y$ such that 
\begin{equation}  \label{anex_una}
	X \circ \MK_n - D \MK_n \cdot Y  = \MG_n,
\end{equation}
with $\MG_n(u, \theta) = (O(u^{n+2}), \,O(u^{n+3}), O(u^{n+2}))$.

Then  we look for  a  $C^1$ function,  $\Delta = \Delta (u, \theta)$, $\Delta: [0, \, \rho)  \times \T^d   \to \R^2$, analytic in $(0, \, \rho)  \times \T^d $, satisfying
\begin{equation} \label{anex_dues} 
	\ X \circ ( \MK_n + \Delta) - D (\MK_n + \Delta) \cdot Y = 0.
\end{equation}
Moreover, we ask $\Delta$ to  be of the form $\Delta = (\Delta^x, \Delta^y, \Delta^\theta) = (O(u^n), \, O(u^{n+1}),  \, O(u^{n}))$.

Using \eqref{anex_una}  we  can rewrite \eqref{anex_dues} as 
\begin{equation} \label{funcional_anex}
	D \Delta \cdot Y =  X \circ ( \MK_n + \Delta  )  -  X \circ \MK_n - \MG_n,
\end{equation}
which is the functional equation that needs to be solved.

We fix $0< \beta <\pi$ and we consider the sector $S(\beta, \rho)$ for some $0 < \rho < 1$. We take the Banach spaces, for $n \in \N$, defined as
\begin{align*}
	\mathcal{Z}_{n} & = \bigg{\{} f : S \times \T^d_\sigma   \rightarrow \C \ | \ f \text{ real analytic,}  \  \|f\|_n: = \sup_{(u, \theta) \in S \times \T^d_\sigma } \frac{|f(u, \theta)|}{|u|^n} < \infty \bigg{\}},
\end{align*}
and we set equation \eqref{funcional_anex} in the ball
$	\mathcal{B}_\alpha \subset    \MZ_n^\times = \MZ_n \times \MZ_{n+1}  \times \MZ_n$, 
endowed with the product norm. 

We define the operators  $\MS_n : \MZ_n \to \MZ_n$ and $\MN_n : \MZ_{n}^\times \to \MZ_{n+1}^\times$ analogously  as in Definitions \ref{def_S_tor_edo}  and \ref{def_N_analitic_tor_edo}, and we obtain the  bounds 
$$
\| \MS_n^ {-1} \| \leq \frac{1}{\mu \, n}, \qquad \mu \in (0, 2 \ol Y_2^x  \cos(\beta/2)), \qquad  n \geq 1,
$$
and 
\begin{align*}
	&\Lip  \MN_n^x \leq  \sup_{\theta \in \T^d_{\sigma}} |c(\theta)| + M_n \rho, \\
	&\Lip \MN_n^y \leq  \max\{\sup_{\theta \in \T^d_{\sigma}} |b(\theta)|, \sup_{\theta \in \T^d_{\sigma}} b(\theta)/2c(\theta) \} + M_n \rho, \\
	&\Lip  \MN_n^\theta \leq \sup_{\theta \in \T^d_{\sigma}} |d(\theta)| + M_n \rho. 
\end{align*}
Finally, we have that $\MT_n = \MS_n^{-1} \circ \MN_n$ is contractive in $\mathcal{B}_\alpha$, which provides a solution, $\Delta$, to \eqref{anex_dues} and concludes the proof of Theorem \ref{teorema_analitic_helicoure}.

	\subsection{Unstable manifolds} \label{subsec_inestable_tors}

The results of this paper concern the existence on stable invariant manifolds. However, completely analogous results to Theorems \ref{teorema_analitic_tors} and \ref{teorema_analitic_tors_edo} hold true for the existence of unstable manifolds assuming that $\ol R_k^x >0$ and $\ol Y_K^x>0$, respectively. Moreover, a formal approximation $\MK_n$ obtained in Theorems \ref{prop_simple_tors} and \ref{prop_tors_edo_sim}, with $\ol R_k^x >0$ and $\ol Y_K^x>0$, respectively, is an approximation of a parameterization of a true unstable manifold.

The results of existence of unstable manifolds for vector fields are obtained by just revesing time $t \mapsto -t$ and applying the results of existence to the new vector field, as we already have done in the applications in Section \ref{sec-appl-tors}.

For the case of maps, if F satisfies the hypotheses of Theorem \ref{teorema_analitic_tors}, the results for the unstable manifolds are obtained from the stated theorem without having to compute explicitely the inverse map $F^{-1}$. We show it in this section.

	To clarify the notation, we will refer to $\MK_n^-$ and $\MR_n^-$ as  approximations of the parametrizations obtained in Proposition \ref{prop_simple_tors} corresponding to the stable manifold and the restricted dynamics on it (namely, with $\ol R_k^x <0$), and to $\MK_n^+$ and $\MR_n^+$ as the parameterizations of the unstable manifold and the restricted dynamics inside it (with $\ol R_k^x >0$).
	
Next we show that the approximation $\MK_n^+$ provided in Proposition \ref{prop_simple_tors} is an approximation of a parameterization of a true unstable manifold, $\wh K^+$, of $F$, asymptotic to $\MT^d$. Moreover, the dynamics on $\wh K^+$ can be parameterized by a map $\wh R^+$ that is also approximated by $\MR_n^+$.  As in the stable case, such pairs of maps also satisfy 
	$$
	\wh K^+(t, \theta) - \MK_n^+ (t, \theta) =  (O(t^{n+1}), O(t^{n+k}), O(t^{n+2p-k})),
	$$
	and 
	$$
	\wh R^+ (t, \theta) - \MR_n^+ (t, \theta ) =  \left\{ \begin{array}{ll}
		(O(t^{2k-1}), 0) & \text{ if }\; n\le k, \\
		(0,0) & \text{ if } \; n>k.
	\end{array} \right.
	$$
	Assume we have a map of the form \eqref{forma_normal_tor}. By Proposition \ref{prop_simple_tors}, there exist approximations $\MK_n^+$ and $\mathcal{R}_n^+$ such that
	\begin{equation} \label{aprox_inestable_tor}
		\MG_n = F \circ \MK_n^+- \MK_n^+ \circ \mathcal{R}_n^+ = (O(t^{n+k}), O(t^{n+2k-1}), O(t^{n+2p-1})),
	\end{equation}
	with 
	$$
	\mathcal{R}_n^+(t, \theta)=
	\begin{pmatrix}
		t + \ol R_k^xt^k + O(t^{k+1})  \\
		\theta + \omega
	\end{pmatrix} 
	$$ and $\ol R_k^x >0$, which means that $\mathcal{R}_n^+$ is a repellor in the normal directions of  $\MT$. Also, $\mathcal{R}_n^+$ is invertible and we have 
	$$
	(\mathcal{R}_n^+)^{-1}(t, \theta) = \begin{pmatrix}
		t - \ol R_k^x t^k + O(t^{k+1}) \\
		\theta - \omega
	\end{pmatrix},
	$$
	and
	$$
	F^{-1}\begin{pmatrix} x\\ y \\ \theta \end{pmatrix}
	= \begin{pmatrix}
		x - c(\theta - \omega)y +  c(\theta - \omega)a_k(\theta - \omega) (x-c(\theta - \omega)y)^k + \check A(x, y, \theta)\\
		y - a_k(\theta - \omega) (x-c(\theta - \omega)y)^k  + \check B(x, y, \theta) \\
		\theta - \omega - d_p(\theta - \omega)(x - c(\theta - \omega)y)^p + \check C(x, y, \theta)
	\end{pmatrix},
	$$
	with $\check A (x, y, \theta), \check B (x, y, \theta) = O(\|(x, y)\|^{k+1}) +yO(\|(x, y)\|^{k-1})$, and $\check C (x, y, \theta) = O(\|(x, y)\|^{p+1}) + yO(\|(x, y)\|^{p-1})$.
	
	Composing by $F^{-1}$ by the left in \eqref{aprox_inestable_tor} and using Taylor's Theorem, we get
	\begin{align} \begin{split} \label{taylor_unstable}
			\MK_n^+ & = 	F^{-1} \circ (\MK_n^+ \circ \MR_n^+ + \MG_n ) \\
			& = F^{-1} \circ (\MK_n^+ \circ \MR_n^+) + DF^{-1} \circ  (\MK_n^+ \circ \MR_n^+) \cdot \MG_n  + O(\MG_n^2),
	\end{split} \end{align}
	and then composing by $(\mathcal{R}_n^+)^{-1}$ by the right we obtain
	\begin{equation} \label{moviment_omega}
		F^{-1} \circ  \MK_n^+ - \MK_n^+ \circ (\mathcal{R}_n^+)^{-1} = (O(t^{n+k}), O(t^{n+2k-1}), O(t^{n+2p-1})).
	\end{equation}
There exists a change of variables $\phi$ that transforms  
$F^{-1}$ into  $H := \phi^{-1} \circ F^{-1} \circ \phi$ which is of the form 
	$$
	H \begin{pmatrix} x\\ y \\ \theta \end{pmatrix}
	= \begin{pmatrix}
		x + c(\theta )y \\
		y + a_k(\theta) x^k  + \check D(x, y, \theta) \\
		\theta - \omega - d_p(\theta )x^p + \check G(x, y, \theta)
	\end{pmatrix},
	$$
	where $\check D(x, y, \theta)$ and $\check G (x, y, \theta)$ satisfy the properties \eqref{condicions_fnormal}.
	Note that he map $H$ is of the form \eqref{forma_normal_tor} (with $\omega$ of opposite sign). Moreover, composing by $\phi^{-1}$ by the left in  \eqref{moviment_omega} and using Taylor's Theorem as in \eqref{taylor_unstable}, we get 
	$$
	\phi^{-1} \circ F^{-1} \circ \MK_n^+ -  \phi^{-1} \circ \MK_n^+ \circ (\MR_n^+)^{-1} =(O(t^{n+k}), O(t^{n+2k-1}), O(t^{n+2p-1})),
	$$
	which is equivalent to 
	\begin{align*}
		H  \circ \phi^{-1} \circ \MK_n^+ - \phi^{-1} \circ \MK_n^+ \circ (\mathcal{R}_n^+)^{-1} =  (O(t^{n+k}), O(t^{n+2k-1}), O(t^{n+2p-1})).
	\end{align*}
	Hence, $H,  \, \phi^{-1}\circ \MK_n^+$ and $(\MR_n^+)^{-1}$ are analytic maps that satisfy the hypotheses of Theorem \ref{teorema_posteriori_tors}, where here the vector of frequencies  is $-\omega$. Therefore, by Theorem \ref{teorema_posteriori_tors}, there exist a map $K^+:[0,\rho) \times \T^d \to \R^2 \times \T^d$,  analytic   in $(0, \rho) \times \T^d$, and an analytic  map $R^+: (-\rho, \rho)\times \T^d \to \R \times \T^d$  such that 
	\begin{equation} \label{invariancia_G_tor}
		H \circ K^+ = K^+ \circ R^+, 
	\end{equation}
	and moreover it holds that 
	\begin{equation} \label{apr_ca_inestable_tor}
		K^+ (t, \theta ) -  \phi^{-1} \MK_n^+(t, \theta ) =  (O(t^{n+1}), O(t^{n+k}), O(t^{n+2p-k})) ,
	\end{equation}
	\begin{equation} \label{erra_invers}
		R^+(t, \theta) - (\mathcal{R}_n^+)^{-1}(t, \theta) = \left\{ \begin{array}{ll}
			(O(t^{2k-1}), 0) & \text{ if }\; n\le k, \\
			(0,0) & \text{ if } \; n>k.
		\end{array} \right.
	\end{equation}
	Also,  composing by $\phi$ by the left in \eqref{invariancia_G_tor} we have 
	$$
	F^{-1} \circ \phi \circ K^+ = \phi \circ K^+ \circ R^+,
	$$
	which means that $\phi \circ K^+$ is a parameterization of a stable manifold of $F^{-1}$, and the restricted dynamics on this stable manifold is given by the map $R^+$, which, using \eqref{erra_invers}, is of the form 
	\begin{equation} \label{erra_2}
		R^+(t, \theta) = 
		\begin{pmatrix}
			t - \ol R_k^x t^k + O(t^{k+1}) \\
			\theta - \omega
		\end{pmatrix},
	\end{equation}
	with $\ol R_k^x >0$.
	
	As a consequence, $\phi \circ K^+$ is a parameterization of an unstable manifold of $F$, analytic in $(0, \rho) \times \T^d$, for some $\rho >0$.  Moreover,  composing by $\phi$ in \eqref{apr_ca_inestable_tor} and using Taylor's Theorem, we have 
	$$
	\phi(K^+(t, \theta))   - \MK_n^+(t, \theta) = (O(t^{n+1}), O(t^{n+k}), O(t^{n+2p-k})), 
	$$
	that is, $\phi \circ K^+$ is approximated by the parameterization $\MK_n^+$ obtained in Proposition \ref{prop_simple_tors}. Denoting $\wh K^+ := \phi \circ K^+$ we recover the notation used at the beginning of the section. 
	
	Finally,  note that since $R^+$ represents the restricted dynamics of $F^{-1}$ on the stable manifold $\phi \circ K^+$, then $(R^+)^{-1}$ represents the restricted dynamics of $F$ on the unstable manifold $\phi \circ K^+$. By the form of \eqref{erra_2} we have
	$$
	(R^+)^{-1}(t, \theta) = 
	\begin{pmatrix}
		t + \ol R_k^x t^k + O(t^{k+1}) \\
		\theta + \omega
	\end{pmatrix},
	$$
	with $\ol R_k^x >0$, and hence, finally, 
	\begin{equation*}
		(R^+)^{-1}(t, \theta) - \mathcal{R}_n^+(t, \theta) = \left\{ \begin{array}{ll}
			(O(t^{2k-1}), 0) & \text{ if }\; n\le k, \\
			(0,0) & \text{ if } \; n>k,
		\end{array} \right.
	\end{equation*}
	as we claimed at the beginning of the section.  Concretely, we recover the notation given at the beginning of the section denoting $\wh R^+ := (R^+)^{-1}$.
		
		\section*{Acknowledgements} 
		
	C. Cufí-Cabré has been supported by the Spanish Government grants MTM2016-77278-P (MINECO/ FEDER, UE), PID2019-104658GB-I00 (MICINN/FEDER, UE)  and BES-2017-081570, and by the Catalan Government grant 2021-SGR-00113. E. Fontich has been supported by the
grant PID2021-125535NB-I00 (MICINN/\-FEDER,UE).
This work has been also funded through the Severo Ochoa and María de Maeztu Program for Centers and Units of Excellence in R\&D (CEX2020-001084-M).

\end{document}